\documentclass{elsarticle}

\usepackage{lineno,hyperref}
\modulolinenumbers[5]

\journal{}

\usepackage{amsmath}
\usepackage{graphicx}
\usepackage{amsmath}
\usepackage{psfrag}
\usepackage{booktabs}
\usepackage{subfigure}
\usepackage{graphicx}
\usepackage{latexsym}
\usepackage{amsmath}
\usepackage{amssymb}
\usepackage{amsfonts}
\usepackage{verbatim}
\usepackage[latin1]{inputenc}
\usepackage{mathrsfs}
\usepackage{amsfonts}
\usepackage{graphicx}
\usepackage{colortbl,dcolumn}
\usepackage{amsmath}
\usepackage{psfrag}
\usepackage{color}
\newtheorem{tm}{Theorem}[section]
\newtheorem{rk}{Remark}[section]

\newtheorem{prop}{Proposition}[section]
\newtheorem{lm}{Lemma}[section]
\newtheorem{cor}{Corollary}[section]

\usepackage{enumerate}

\newcommand{\E}{\mathbb E}
\newcommand{\PP}{\mathbb P}
\newcommand{\N}{\mathbb N}
\newcommand{\R}{\mathbb R}

\newcommand{\Z}{\mathbb Z}
\newcommand{\bi}{\mathbf i}

\newcommand{\C}{\mathcal C}
\newcommand{\OO}{\mathcal O}
\newcommand{\PPP}{\mathcal P}
\newcommand{\OOO} {\mathscr O}
\newcommand{\HH}{\mathbb H}
\newcommand{\LL}{\mathcal L}

\newcommand{\FFF}{\mathscr F}
\newcommand{\ee}{\varepsilon}

\newcommand{\<}{\langle}
\renewcommand{\>}{\rangle}

\begin{document}
	
	\begin{frontmatter}
		
		\title{Strong Convergence Rate of Splitting Schemes for Stochastic Nonlinear Schr\"odinger Equations}
		\tnotetext[mytitlenote]{This work was supported by National Natural Science Foundation of China (No. 91630312, No. 91530118,  and No. 11290142).} 
		
		\author[cas]{Jianbo Cui\corref{cor}}
		\ead{jianbocui@lsec.cc.ac.cn}
		
		\author[cas]{Jialin Hong}
		\ead{hjl@lsec.cc.ac.cn}
		
		\author[cas]{Zhihui Liu}
		\ead{liuzhihui@lsec.cc.ac.cn}
		
		\author[cas1]{Weien Zhou}
		\ead{weienzhou@nudt.edu.cn}
		
		\cortext[cor]{Corresponding author.}
		
		\address[cas]{1. LSEC, ICMSEC, 
			Academy of Mathematics and Systems Science, Chinese Academy of Sciences, Beijing,  100190, China\qquad
			2. School of Mathematical Science, University of Chinese Academy of Sciences, Beijing, 100049, China 
			}
		\address[cas1]{College of Science, National University of Defense Technology, Changsha, China}

	\begin{abstract}
		In this paper, we show that solutions of  stochastic nonlinear Schr\"odinger (NLS) equations can be approximated by  solutions of coupled splitting systems.
		Based on these systems,  we propose a new kind of  fully  discrete splitting schemes which possess algebraic strong convergence rates for stochastic  NLS equations. Key ingredients of our approach are the exponential integrability and stability of the corresponding 
splitting systems and numerical approximations.
In particular, under very mild conditions, we derive the optimal strong convergence rate $\OOO(N^{-2}+\tau^\frac12)$ of the spectral splitting Crank--Nicolson scheme, where $N$ and $\tau$ denote the dimension of the approximate space and the time step size, respectively.
	\end{abstract}
		

	\end{frontmatter}

\section{Introduction}
For stochastic partial differential equations (SPDEs) with monotone coefficients, there exist fruitful results on strong error analysis of temporal and/or spatial numerical approximations by using semigroup or variational frameworks (see e.g. \cite{AC16,  CHL17, CHL16a, dBD06, JK10, JR15, KLL15, Yan05}).
However, for SPDEs with non-monotone coefficients,
so far as we know, it is a long standing open problem to construct temporal numerical approximations and full discretizations which possess algebraic strong convergence rates. 
This is the main motivation of the present paper.

As a classical type of SPDEs with non-monotone coefficients, stochastic NLS equations model the propagation of nonlinear dispersive waves in inhomogeneous or random media (see e.g. \cite{KV94} and references therein). 
Our main purpose in this paper is to construct temporal approximations and fully discrete schemes possessing algebraic strong convergence rates for the one-dimensional stochastic NLS equation
\begin{align}\label{nls}
\begin{cases}
\bi du+(\Delta u+\lambda |u|^2 u) dt=u\circ dW(t),
\quad \text{in}\quad  (0,T]\times \OOO; \\
u(t)=0,\quad \text{on}\quad  [0,T]\times \partial \OOO; \\
u(0)=\xi,\quad \text{in}\quad \OOO,
\end{cases}
\end{align}
where $T>0$, $\OOO=(0, 1)$, and $\lambda=1$ or $-1$ corresponds to focusing or defocusing cases, respectively.
Here $W=\{W(t):\ t\in [0,T]\}$ is a $\mathbb L^2(\OOO; \mathbb R)$-valued $Q$-Wiener process on  a stochastic basis $(\Omega, \FFF, \FFF_t, \PP)$, i.e., there exists an orthonormal basis $\{e_k\}_{k\in \N_+} $ of $\mathbb L^2(\OOO; \mathbb R)$ and a sequence of mutually independent, real-valued Brownian motions $\{\beta_k\}_{k\in \N_+} $ such that $W(t)=\sum_{k\in \N_+} Q^{\frac12}e_k\beta_k(t)$, $t\in [0,T]$.

Eq. \eqref{nls} has been investigated both theoretically and numerically.
The well-posedness of Eq. \eqref{nls} has been proved by \cite{dBD03} in $\HH^1$, by \cite{CHP16} and \cite{CHL16b} in $\HH^2$ for defocusing and focusing cases, respectively.
There are also many authors constructing numerical approximations of Eq. \eqref{nls} and obtaining convergence rates in certain sense such as pathwise or in probability weaker than in strong sense (see e.g. \cite{ CHP16, CD17,  CHLZ17, dBD06, Liu13b} and references therein). 
A progress has been made by \cite{CHL16b} where the authors obtain a strong convergence rate of spatial centered difference method for Eq. \eqref{nls}.
Besides the $\HH^2$-a priori estimations, the key ingredient to derive strong convergence rates is the $\HH^1$-exponential integrability of both the exact and numerical solutions (see also \cite{HJ14, HJW16}).
This type of exponential integrability is also useful to get the strongly continuous dependence (in $\mathbb L^p(\Omega; \C([0,T]; \mathbb L^2))$) on initial data of both the exact and numerical solutions, to derive a large deviation principle of  Freidlin--Wentzell type (see \cite[Corollaries 3.1 and 3.2]{CHL16b}) and to deduce Gaussian tail estimations of these solutions (see Corollary \ref{tail}). 
We refer to \cite{HJ14, HJW16, JP16} and references therein for the exponential integrability of a kind of  stochastic evolution   equations  with non-monotone coefficients and of their numerical approximations. 
So far as we know, there exists no result about this type of exponential integrability for a temporally discrete approximation of Eq. \eqref{nls}.
In this work we propose a temporal splitting Crank--Nicolson scheme (see \eqref{spl-cn}), based on a splitting approach and its corresponding splitting processes which are shown to admit the desired exponential integrability.

Let us loosely describe the achievement of a sequence of splitting processes through the splitting approach.
Given an $M\in \N$, denote $\Z_M=\{0,1,\cdots,M\}$.
Let $\tau=\frac T{M}$ and $\{T_m:=(t_m, t_{m+1}],\ t_m=m\tau,\ m\in \Z_{M-1}\}$ be a uniform partition of the interval $(0,T]$.
Our main idea is to split Eq. \eqref{nls} in $T_m$, $m\in \Z_{M-1}$, into a deterministic NLS equation with random initial datum and a linear SPDE:
\begin{align}\label{nls-d}
du_{\tau,m}^D(t)
&=\left(\bi \Delta u_{\tau,m}^D(t)
+\bi \lambda |u_{\tau,m}^D(t)|^2 u_{\tau,m}^D(t) \right)dt,
\quad u_{\tau,m}^D (t_m)=u_\tau(t_m),\\
\label{nls-s}
du_{\tau,m}^S(t)
&=-\bi  u_{\tau,m}^S(t) \circ dW(t),
\quad u_{\tau,m}^S (t_m)=u^D_{\tau,m}(t_{m+1}).
\end{align}
Eq. \eqref{nls-d} and \eqref{nls-s} are subjected to homogeneous Dirichlet boundary conditions on $T_m\times \partial \OOO$.
Then, in Section \ref{sec-spl}, we give the definition of a auxiliary  splitting process $u_\tau$ with  the initial datum  $u_\tau(0)=\xi$.
It is shown that this splitting process 
$u_\tau=\{u_\tau(t):\ t\in [0,T]\}$, 
which is left-continuous with finite right-hand limits and $\FFF_t$-adapted (see Proposition \ref{cha-ene-spl}).
Since Eq. \eqref{nls-d} has no analytic  solution, we apply the Crank--Nicolson scheme to temporally discretize Eq. \eqref{nls-d}.
Based on the explicitness of the solution of Eq. \eqref{nls-s}, we get the splitting Crank--Nicolson scheme starting from $\xi$:
\begin{align}\label{splcn}
\begin{cases}
u^D_{m+1}
= u_m+ \bi \tau \Delta  u^D_{m+\frac12}+\bi \lambda \tau \frac {| u_m|^2+| u_{m+1}^D|^2}2  u^D_{m+\frac12},\\
u_{m+1}=\exp\left(-\bi (W_{t_{m+1}}-W_{t_m})\right) u_{m+1}^D,
\quad m\in \Z_M,
\end{cases}
\end{align}
with $ u_{m+\frac12}^D=\frac12(u_m+ u_{m+1}^D)$.

Our first goal is to prove that both $u_\tau=\{u_\tau(t):\ t\in [0,T]\}$ and $\{u_m\}_{m\in \Z_{M}}$ converge to the exact solution $u=\{u(t):\ t\in [0,T]\}$ of Eq. \eqref{nls} with strong order $1/2$ (see Theorem \ref{u-ut} and Theorem \ref{u-um}).
The key requirement is the exponential integrability properties of $u_\tau$ and $u_m$, which is proved by an exponential integrability lemma established in \cite[Corollary 2.4]{CHJ13} or \cite[Proposition 3.1]{CHL16b} (see Lemmas \ref{exp-ut} and \ref{exp-um}). 
To the best of our knowledge, Theorem \ref{u-um} is the first result about strong convergence rates of  temporal approximations for Eq. \eqref{nls} or even for SPDEs with non-monotone coefficients.
We also note that there are several results to numerically approximate SPDEs by splitting schemes (see e.g. \cite{BBM14, Dor12, Liu13a, Liu13b, GK03} and references therein).
\cite{GK03} obtains a strong convergence rate of a splitting scheme for a linear SPDE about stochastic filtering problem;
\cite{Liu13a} gets a strong convergence rate for a linear stochastic Schr\"odinger equation.
However, the splitting schemes in \cite{GK03, Liu13a} applied to Eq. \eqref{nls} would fail to satisfy the desired exponential integrability.

Our second goal is to construct a fully discrete scheme possessing  optimal algebraic strong convergence rate based on the aforementioned splitting approach.
To this end, we apply the splitting Crank--Nicolson scheme \eqref{splcn} to the spatially spectral Galerkin discretization in Section \ref{sec-ful} and get the spectral splitting Crank--Nicolson scheme \eqref{spe-cn}.
The spectral approximate solution $u^N$ is shown to converge to $u$ with strong convergence rate $\OOO(N^{-2})$, where $N$ is the dimension of the spectral approximate space.
This convergence rate is optimal under minimal assumptions on the initial datum $\xi$ and the noise's covariance operator $Q$.
Combining the strong error estimate of the splitting Crank--Nicolson scheme \eqref{splcn}, we finally derive the strong convergence rate $\OOO(N^{-2}+\tau^\frac12)$ of this fully discrete scheme (see Theorem \ref{u-umn}).
 We remark that there exist a lot of alternative 
choices of spatial discretizations. 
For instance, we apply this splitting approach to the spatial centered difference method analyzed in \cite[Theorem 4.1]{CHL16b} and get the related full discretization with algebraic strong convergence rate.

Our article is organized as follows.
We give detailed analysis for the splitting process 
$u_\tau$ in Section \ref{sec-spl}.
In this section, we study the evolutions about the charge, energy and a Lyapunov functional used to control the $\HH^2$-norm of $u_\tau$ as well as its  exponential integrability in $\HH^1$.
Based on the $\HH^2$-a priori estimation and exponential integrability of  $u_\tau$ and the exact solution $u$, we deduce the strong error estimate between $u_\tau$ and $u$.
In Section \ref{sec-cn}, we analyze the temporal splitting Crank--Nicolson scheme and obtain its strong convergence rate.
To perform an implementary full discretization, we apply the proposed splitting Crank--Nicolson scheme to the spectral Galerkin approximate equation in Section \ref{sec-ful} and get the strong convergence rate of this spectral splitting Crank--Nicolson scheme for Eq. \eqref{nls}.
Similar arguments are also applied to spatially discrete equation by centered difference method.

To close this section, we introduce some frequently used notations and assumptions.
The norm and inner product of $\mathbb L^2:=\mathbb L^2(\OOO; \mathbb C)$ is denoted by $\|\cdot\|$ and $\<u, v\>:=\Re\left[\int_{\OOO} \overline u(x) v(x) dx\right]$, respectively.
Throughout we assume that $T$ is a fixed positive number, $\xi \in \HH_0^1\cap \HH^2$ is a deterministic function and $Q^\frac12 \in \LL_2^2= \LL^2(\HH,  \HH_0^1\cap \HH^2)$, i.e.,
\begin{align*}
\|Q^\frac12\|_{\LL_2^2}^2
:=\sum_{k\in \N_+}  \|Q^\frac12 e_k\|_{\HH^2}^2<\infty,
\end{align*}
where $\{e_k\}_{k\in \N_+} $ is any orthonormal basis of $\mathbb L^2(\OOO; \R)$.
We use $C$ and $C'$ to denote a generic constant, independent of the time step size 
$\tau$ and the dimension $N$, which differs from one place to another.

\section{Splitting Process}
\label{sec-spl}

We first give the definition of the auxiliary splitting processes $u_\tau(t)$ and $u_\tau^D(t)$,  $t\in [0,T]$.
For simplicity, we denote the solution operators of Eq. \eqref{nls-d} and \eqref{nls-s} in $T_m$ as $\Phi_{m,t-t_m}^D$ and $\Phi_{m,t-t_m}^S$, respectively.
Next we set the splitting process $u_\tau$ in $T_m$ as 
\begin{align}\label{spl}
u_\tau (t):=u_{\tau,m}^S(t):=(\Phi_{j,t-t_m}^S\Phi_{j,\tau}^D)\prod_{j=1}^{m-1}\big(\Phi_{j,\tau}^S\Phi_{j,\tau}^D\big)u_\tau(0),\quad t\in T_m.
\end{align}
and
\begin{align*}
u_{\tau}^D(t):=u_{\tau,m}^D(t):=\Phi_{j,t-t_m}^D\prod_{j=1}^{m-1}\big(\Phi_{j,\tau}^S\Phi_{j,\tau}^D\big)u_\tau(0),\quad t\in  [t_m,t_{m+1}).
\end{align*}
Our main purpose is to prove that $u_\tau$ possesses the exponential integrability and is a nice approximation of the exact solution $u$ of Eq. \eqref{nls}. 

We first recall the following known results about the well-posedness and strongly continuous dependence on initial  data for Eq. \eqref{nls} as well as exponential integrability of $u$. 
These properties are used to derive algebraic strong convergence rates for numerical approximations of Eq. \eqref{nls}.

\begin{tm} \label{well}
Let $p\ge 1$.
Eq. \eqref{nls} possesses a unique strong solution $u=\{u(t):\ t\in [0,T]\}$ satisfying 
\begin{align}\label{sta-u}
\E\left[\sup_{t\in [0,T]}\|u(t)\|^p_{\HH^2}\right]< \infty
\end{align}
and depending on the initial data in $\mathbb L^p(\Omega; \C([0,T]; \mathbb L^2))$, i.e., if we assume that $\xi^I$, $\xi^{II}\in \HH_0^1\cap \HH^2$ and that $u^I$ and $u^{II}$ are the solutions of Eq. \eqref{nls} with initial data $\xi^I$ and $\xi^{II}$, respectively, then there exists a constant $C=C(\xi^I,\xi^{II},Q,T,p)$ such that 
\begin{align}\label{con-dep}
\left(\E\left[\sup_{t\in [0,T]}\|u^I(t)-u^{II}(t)\|^p \right]\right)^\frac1p
\le C \|\xi^I-\xi^{II}\|.
\end{align}
Moreover, there exist constants $C$ and $\alpha$ depending on $\xi$, $Q$ and $T$ such that
\begin{align}\label{exp-u}
\sup_{t\in[0,T]}\E\left[\exp\left( \frac{\|\nabla u(t)\|^2}{e^{\alpha t}}\right)\right]
&\le C.
\end{align}
\end{tm}

\textbf{Proof}
We refer to \cite{CHL16b}, Theorem 2.1, Corollary 3.1 and Proposition 3.1,  respectively, 
for the well-posedness and $\HH^2$-a priori estimate estimate \eqref{sta-u}, strongly continuous dependence estimate \eqref{con-dep} and exponential integrability estimate \eqref{exp-u}. 
\qed

\subsection{Stability of Splitting Process}
\label{sec-spl-h2}

In this part, we prove that the splitting  process $u_\tau=\{u_\tau(t):\ t\in [0,T]\}$ defined by \eqref{spl} is well-defined and uniformly bounded in $\mathbb L^p(\Omega; \C([0,T]; \HH^2))$ for any $p\ge 1$. 

We start with the evolution of the charge and  the energy
of $u_\tau$, i.e., $\|u_\tau\|^2$ and $H(u_\tau):=\frac12 \|\nabla u_\tau\|^2-\frac\lambda4 \|u_\tau\|_{L_4}^4$.

\begin{prop}\label{cha-ene-spl}
The splitting process $u_\tau=\{u_\tau(t):\ t\in [0,T[\}$ is uniquely solvable and $\FFF_t$-measurable.
Moreover, for any $t\in [0,T]$ there holds a.s. that 
\begin{align}\label{cha-spl}
\|u_\tau(t)\|^2=\|\xi\|^2
\end{align}
and that 
\begin{align}\label{ene-spl}
H(u_\tau(t))
&=H(\xi)+\int_0^t \left\<\nabla u_\tau, \bi u_\tau d(\nabla W(r))\right\> 
\nonumber  \\
&\quad+\frac12 \sum_{k\in \N}\int_0^t \|u_\tau \nabla(Q^\frac12 e_k)\|^2 dr.
\end{align}
\end{prop}

\textbf{Proof}
Let $m\in \Z_M$ and $t\in T_m$.
Since Eq. \eqref{nls-d} can be seen as a special equation of \eqref{nls} with $Q=0$,  $u_\tau^D(t)$ is uniquely solvable and $\FFF_{t_m}$-measurable by Theorem \ref{well}.
Moreover, 
\begin{align}\label{cha-ene-ud}
\|u_\tau^D(t)\|^2=\|u_{\tau,m}^D(t_m)\|^2,\quad 
H(u_\tau^D(t))=H(u_{\tau,m}^D(t_m))\quad \text{a.s.}
\end{align}
On the other hand, Eq. \eqref{nls-s} has an $\FFF_t$-measurable analytic solution given by $u_\tau^S(t)=\exp\left(-\bi (W(t)-W(t_m))\right) u_{\tau,m}^S(t_m)$,
and thus $u_\tau^S(t)$ preserves the modular length as well as the charge, i.e.,  
\begin{align}\label{cha-us}
|u_\tau^S(t)|=|u_{\tau,m}^S(t_m)|,\quad
\|u_\tau^S(t)\|^2=\|u_{\tau,m}^S(t_m)\|^2.
\end{align}
Then $u_\tau$ is uniquely solvable and $\FFF_t$-measurable and 
\begin{align*}
\|u_\tau(t)\|^2=\|u_\tau(t_m)\|^2,
\end{align*}
from which we obtain \eqref{cha-spl} by iterations on $m$.
By It\^o formula, we have 
\begin{align}\label{ene-us}
H(u_\tau(t))
&=H(u_\tau(t_m))+\int_{t_m}^t \left\<\nabla u_\tau,\bi u_\tau d(\nabla W(r)) \right\>  \nonumber \\  
&\quad+\frac12 \sum_{k\in \N}\int_{t_m}^t 
\|u_\tau \nabla(Q^\frac12 e_k)\|^2 dr.
\end{align}
Combining \eqref{cha-ene-ud} and \eqref{ene-us}, we obtain \eqref{ene-spl} by iterations. 
\qed

The above charge conservation law \eqref{cha-spl} and energy evolution \eqref{ene-spl} imply the following boundedness in $\mathbb L^p(\Omega; \C([0,T]; \HH^1))$ for any $p\ge 1$.

\begin{cor}\label{h1-sta}
For any $p\ge 1$, there exists a constant $C=C(\xi,Q,T, p)$ such that
\begin{align}\label{h1-sta0}
\E\left[\sup_{t\in [0,T]}\|u_\tau(t)\|^p_{\HH^1}\right]\le C.
\end{align}
\end{cor}

\textbf{Proof}
Let $t\in T_m$ for $m\in \Z_M$ and $p\ge 4$.
Applying  It\^o formula and using the energy evolution law \eqref{ene-us} of $u^S_\tau$, we obtain
\begin{align*}
H^{\frac p2}(u_{\tau,m}^S(t))
&=H^{\frac p2}(u_{\tau,m}^S(t_m))
+\frac p4\sum_{k\in \N} \int_{t_m}^t H^{{\frac p2}-1}(u_{\tau,m}^S) 
\|u_{\tau,m}^S \nabla(Q^\frac12 e_k)\|^2 dr \\
&\quad +\frac p4\int_{t_m}^t H^{{\frac p2}-1}(u_{\tau,m}^S) 
\left\<\nabla u_{\tau,m}^S,\bi u_{\tau,m}^S d(\nabla W(r)) \right\>  \\ 
&\quad+\frac{p(p-2)}8\sum_{k\in \N} \int_{t_m}^t H^{{\frac p2}-2}(u_{\tau,m}^S) 
\left\<\nabla u_{\tau,m}^S,\bi u_{\tau,m}^S \nabla (Q^\frac12 e_k) \right\> dr.
\end{align*}
Taking expectations on both sides of the above equality, using H\"older and Gagliardo--Nirenberg inequalities, we get
\begin{align*}
\E \left[H^{\frac p2} (u_{\tau,m}^S(t))\right]
&\le \E \left[H^{\frac p2} (u_{\tau}(t_m))\right]
+C\int_{t_m}^t \left(1+\E \left[H^{\frac p2} (u_{\tau,m}^S(s))\right] \right)ds.
\end{align*}
Gronwall inequality yields that 
\begin{align}\label{ene-p}
\E \left[H^{\frac p2}(u_{\tau,m}^S(t))\right]
\le e^{C\tau} \left(\E \left[H^{\frac p2}(u_\tau(t_m))\right]+C\tau \right).
\end{align}
Using the energy conservation law of  Eq. \eqref{nls-d} and substituting iteratively the estimation \eqref{ene-p}  for $\E \left[H^{\frac p2}(u_\tau(t_m))\right]$ in $T_{m-1}$, we have
	\begin{align}\label{ite-hp}
\E \left[H^{\frac p2}(u_\tau(t))\right]
&\le (e^{C\tau})^2 \E\left[H^{\frac p2}(u_\tau(t_{m-1})) \right]+C\tau(1+e^{C\tau})
\le \cdots \nonumber \\
&\le (e^{C\tau})^{m+1} H^{\frac p2}(\xi)+C\tau\sum_{k=0}^m (e^{C\tau})^k \nonumber \\
&\le e^{CT}H^{\frac p2}(\xi)+C\tau\frac{e^{CT}-1}{e^{C\tau}-1}
\le e^{CT}(1+H^{\frac p2}(\xi)).
\end{align}
Therefore,
\begin{align}\label{h1-sta1}
	\sup_{t\in [0,T]}\E\left[\|u_\tau(t)\|^p_{\HH^1}\right]\le C.
	\end{align}
By	\eqref{ene-spl}, we have that 
		\begin{align*}
		\sup_{t\in [0,T]}H(u_\tau(t))
		&\le H(\xi)+\sup_{t\in [0,T]}\left|\int_0^t \left\<\nabla u_\tau(r), \bi u_\tau(r) dW(r)\right\>\right|\\
		&\quad +\frac12 \sum_{k\in \N}\int_0^T \|
		u_\tau(r) \nabla(Q^\frac12 e_k)\|^2 dr.
		\end{align*}
By Burkholder--Davis--Gundy inequality and charge conservation law \eqref{cha-spl}, we get 
 \begin{align*}
\left\|\sup_{t\in [0,T]}\left|\int_0^t \left\<\nabla u_\tau(r), \bi u_\tau(r) dW(r)\right\>\right|\right\|_{\mathbb L^{\frac p2}(\Omega)}
\le C \sup_{t\in [0,T]}\|\nabla u_\tau(r)\|_{\mathbb L^{\frac p2}(\Omega;\mathbb L^2)},
\end{align*}
which is bounded due to \eqref{h1-sta1}.
This in turn implies \eqref{h1-sta0} for $p\ge 4$.
The estimation of \eqref{h1-sta0} for $p\in [1,4)$ follows by H\"older inequality. 
\\\qed

Similar  arguments yields the boundedness of $u_\tau^D$.
\begin{cor}\label{h1-ud}
 The auxiliary process $u_\tau^D$ is the right continuous and $\FFF_t$-adapted.
Moreover,for any $p\ge 1$, there exists a constant $C=C(\xi,Q,T, p)$ such that
	\begin{align}\label{h1-ud0}
	\E\left[\sup_{t\in [0,T]}\|u_\tau^D(t)\|^p_{\HH^1}\right]\le C.
	\end{align}
\end{cor}

Now we can show the $\HH^2$-a priori estimate of $u_\tau$ through the evolution of the Lyapunov functional
\begin{align}\label{h2-lya}
f(u):=\|\Delta u\|^2+\lambda\<\Delta u, |u|^2u\>.
\end{align}

\begin{prop}\label{h2-sta}
For any $p\ge 1$, there exists a constant $C=C(\xi,Q,T, p)$ such that
\begin{align}\label{h2-sta0}
\E\left[\sup_{t\in [0,T]}\|u_\tau(t)\|^p_{\HH^2}\right]\le C.
\end{align}
\end{prop}

\textbf{Proof}
Let $t\in T_m$ for some $m\in \Z_M$ and $p=2$.
By the same arguments in \cite[Theorem 2.1]{CHL16b} (see (2.5) and (2.8) in \cite{CHL16b} with $Q=0$), we have  
\begin{align*}
&\E\left[f(u_\tau^D(t))\right]-\E\left[f(u_\tau^D(t_m)) \right] \\
&=\int_{t_m}^{t} \E\left[\left\<\Delta u_{\tau,m}^D(r),\bi |u_{\tau,m}^D(r)|^4u_{\tau,m}^D(r)\right\>\right] dr\\
&\quad+\lambda \int_{t_m}^{t} \E\left[ \left\<\Delta u_{\tau,m}^D(r),
4\bi |\nabla  u_{\tau,m}^D(r)|^2 u_\tau^D(r)
+2\bi (\nabla  u_{\tau,m}^D(r))^2  \overline{u_{\tau,m}^D(r)}\right\>\right] dr \\
&\le C\Big(1+\sup_{t\in [0,T]}\E\left[ \|u_\tau^D(t)\|_{\HH^1}^{10}\right] \Big) \tau
+C\int_{t_m}^t \E\left[ f(u_{\tau,m}^D(r))\right] dr.
\end{align*}
Combined with the relationship between $H(u_\tau^D)$ and $\|u_\tau^D(t)\|_{\HH^1}$, the a priori estimate \eqref{h1-sta0}  and Gronwall inequality imply that 
\begin{align}\label{fud}
\E\left[f(u_{\tau,m}^D(t))\right]
&\le e^{C\tau}\E\left[f(u_{\tau,m}^D(t_m)) \right]+C\tau.
\end{align}
On the other  hand, applying It\^o formula to $f(u_\tau)$, we obtain 
\begin{align*}
&f(u_\tau(t))-f(u_\tau^D(t_{m+1})) \\
&=2\int_{t_m}^{t}\left\<\Delta u_{\tau,m}^S, 
\Delta\left(-\bi u_{\tau,m}^S dW(r)
-\frac12u_{\tau,m}^S F_Q dr\right) \right\> \\
&\quad +\lambda\int_{t_m}^{t}\left\<\Delta u_{\tau,m}^S,
|u_{\tau,m}^S|^2\Big(-\bi u_{\tau,m}^S dW(r)-\frac12u_{\tau,m}^S F_Q dr\Big)\right\>  \\
&\quad+\lambda\int_{t_m}^{t}\left\<\Delta\Big(-\bi u_{\tau,m}^S dW(r)-\frac12u_\tau^S F_Q dr\Big), |u_{\tau,m}^S|^2u_{\tau,m}^S \right\>  \\
&\quad +2\lambda \int_{t_m}^{t} \left\< \Delta (-\bi u_{\tau,m}^S Q^{\frac12}e_k ), -\bi |u_{\tau,m}^S|^2u_{\tau,m}^S Q^{\frac12}e_k\right\>dr \\
&\quad+\lambda\int_{t_m}^{t}\left\<\Delta u_{\tau,m}^S, -|u_{\tau,m}^S|^2 u_{\tau,m}^S F_Q \right\> dr
+\int_{t_m}^{t}\sum_{k\in \N} \|\Delta (u_{\tau,m}^S Q^{\frac12}e_k)\|^2 dr.
\end{align*} 
Taking expectation and using H\"older and Gagliardo--Nirenberg inequalities and the $\HH^1$-a priori estimate \eqref{h1-sta1}, we obtain 
\begin{align}\label{uts}
\E \left[f(u_\tau(t))\right]
\le \E \left[f(u_\tau^D(t_{m+1}))\right]
+C\int_{t_m}^t \left(1+\E\left[ f(u_{\tau,m}^S)\right]\right) dr.
\end{align}
Then Gronwall inequality implies that 
\begin{align}\label{fus}
\E \left[f(u_\tau(t))\right]
&\le e^{C\tau}\E\left[f(u_\tau^D(t_{m+1}))\right]+C\tau.
\end{align}
Similar iterative arguments to derive \eqref{ite-hp} applying to $\E \left[f(u_\tau^S(t))\right]$, combining with \eqref{fus}, yields 
\begin{align*}
\E \left[f(u_\tau(t))\right]
\le e^{CT}(1+f(\xi)).
\end{align*}
These estimations in turn show that 
\begin{align*}
\sup_{t\in[0,T]}\E \left[ \|u_\tau(t)\|_{\HH^2}^2\right]\le C.
\end{align*}
To derive \eqref{h2-sta0} for $p\ge 4$, one only need to apply It\^o formula to $f^\frac p2(u_\tau(t))$ and Burkholder--Davis--Gundy inequality to the stochastic integral as in Lemma \ref{h1-sta}.
The estimation of \eqref{h2-sta0} for $p\in [1,4)$ follows from H\"older inequality.
We omit the details here to avoid the tedious calculations.  
\qed

\begin{cor}\label{h2-ud}
For any $p\ge 1$, there exists a constant $C=C(\xi,Q,T, p)$ such that
\begin{align}\label{h2-ud0}
\E\left[\sup_{t\in [0,T]}\|u_\tau^D(t)\|^p_{\HH^2}\right]\le C.
\end{align}

\end{cor}

\textbf{Proof}
The estimations \eqref{h2-ud0} follows from \eqref{fud} and iterative arguments.
\qed

\subsection{Exponential Integrability of Splitting Process}
\label{sec-spl-exp}

In this part we prove the $\HH^1$-exponential integrability for $u_\tau=\{u_\tau(t):\ t\in [0,T]\}$.
This property is the key ingredient to derive the strong error estimate between $u_\tau$ and $u$. 

We recall a useful exponential integrability lemma. 

\begin{lm}\label{exp-int}
Let $\HH$ be a  separable Hilbert space, $U\in \C^2(\HH;\R)$, $\overline U$ be a functional in $\HH$ and $X$ be an $\HH$-valued, adapted stochastic process with continuous sample paths satisfying 
$\int_0^T\|\mu(X_s)\|+\|\sigma(X_s)\|^2 dr<\infty$ a.s., and for all $t\in [0,T]$, $X_t=X_0+\int_0^t \mu(X_s)dr +\int_0^t \sigma(X_s)dW_s$ a.s.
Assume that there exists an $\FFF_0$-measurable random variable $\alpha\in [0,\infty)$ such that a.s.
\begin{align*}
&DU(X)\mu(X)
+\frac{\textup{tr}\left[D^2U(X)\sigma(X)\sigma^*(X)\right]}2 \nonumber  \\
&\quad +\frac{\|\sigma^*(X) D U(X)\|^2}{2e^{\alpha t}}+\overline U(X)
\le \alpha U(X),
\end{align*}
then 
\begin{align*}	
\sup_{t\in [0,T]}\E\left[\exp\left( \frac {U(X_t)}{e^{\alpha t}}+\int_0^t\frac {\overline U(X_r)}{e^{\alpha s}}dr \right)\right]
\le \E\left[ e^{U(X_0)} \right].
\end{align*}
\end{lm}

\textbf{Proof}
See \cite[Lemma 3.1]{CHL16b} or \cite[Corollary 2.4]{CHJ13}.
\qed

Based on Lemma \ref{exp-int}, we present the exponential integrability of $u_\tau$. It should be mentioned that
For SPDEs, the existence of the strong solution is not uncommon. However, we can use the finite-dimensional 
approximation and Fatou lemma to rigorously prove the following lemma for the mild solution, under the assumption that $Q\in \LL_2^2$ 
and $\xi \in  \HH_0^1 \cap \HH^2$. 

\begin{lm}\label{exp-ut}
There exist  constants $C$ and $\alpha$ depending on $\xi$, $Q$ and $T$ such that
\begin{align}\label{exp-ut00}
\sup_{t\in[0,T]}\E \left[\exp\left(\frac {H( u_\tau(t))}{e^{\alpha t}}\right)\right]
\le C
\end{align}
and 
\begin{align}\label{exp-ut0}
\sup_{t\in[0,T]}\E \left[\exp\left(\frac {\|\nabla u_\tau(t)\|^2}{e^{\alpha t}}\right)\right]
\le C.
\end{align}
\end{lm}
	
\textbf{Proof}
We first prove \eqref{exp-ut00}.
Since  Eq. \eqref{nls-d} possesses the energy conservation law,
we focus on  Eq. \eqref{nls-s}. For simplicity, we denote $\sigma(u)=-\bi u Q^{\frac12}$, $\mu(u)=-\frac12uF_Q$ and omit the variable $t$ in $u_\tau^S$.
Simple calculations yield that in $T_m$, 
\begin{align*}
&DH(u_\tau)\mu(u_\tau)+\frac{\text{tr}\left[\sigma(u_\tau)\sigma^*(u_\tau)D^2H(u_\tau)\right]}2 
+\frac{\|\sigma^*(u_\tau)DH(u_\tau)\|^2}{2e^{\alpha_{\lambda} (t-t_m)}}\\
&= -\frac12\left\<\nabla u_\tau, u_\tau\nabla F_Q\right\>
+\frac{\sum_{k\in \N} \<\nabla u_\tau, \bi u_\tau Q^{\frac12}e_k\>^2}{2e^{\alpha_{\lambda} (t-t_m)}}.
\end{align*}
We conclude that
\begin{align*}
&DH(u_\tau)\mu(u_\tau)+\frac{\text{tr}\left[\sigma(u_\tau)\sigma^*(u_\tau)D^2H(u_\tau)\right]}{2}  \\
&\quad +\frac{\|\sigma^*(u_\tau)DH(u_\tau)\|^2}{2e^{\alpha_{\lambda} (t-t_m)}}
\le \alpha_{\lambda} H(u_\tau)+\beta_{\lambda}
\end{align*}
with 
\begin{align*}
\alpha_{-1}=2 \|Q^\frac12\|_{\LL_2^2}^2\|\xi\|^2,\quad
\beta_{-1} = \|Q^\frac12\|_{\LL_2^2}^2 \|\xi\|^2
\end{align*}
and
\begin{align*}
\alpha_1=4\|Q^\frac12\|_{\LL_2^2}^2 \|\xi\|^2,\quad 
\beta_1=\|Q^\frac12\|_{\LL_2^2}^2 (\|\xi\|^2+\|\xi\|^8).
\end{align*}
Applying Lemma \ref{exp-int} with $\overline{U}=-\beta_\lambda$ and the energy conservation law of $u_\tau^D$ in $T_m$ , we obtain 
\begin{align*}
\sup_{t\in T_m}\E\left[\exp\left(\frac {H(u_\tau(t))}{e^{\alpha_{\lambda}(t-t_m)}}\right)\right]
&\le e^{\beta_{\lambda}(t-t_m)} \E\left[ e^{H(u_\tau(t_m))}\right].
\end{align*}	
Similar arguments yield that 	
\begin{align*}
\E\left[\exp\left(\frac {H(u_\tau(t))e^{-\alpha_{\lambda}t_m}}{e^{\alpha_{\lambda}(t-t_m)}}\right)\right]
&\le e^{\beta_{\lambda}(t-t_m)}  
\E\left[ \exp\left(e^{-\alpha_{\lambda}t_m}H(u_\tau(t_m))\right)\right].
\end{align*}
Repeating the previous procedure, 
we get 
\begin{align*}
\sup_{t\in T_m}\E\left[\exp\left(\frac {H(u_\tau(t))}{e^{\alpha_{\lambda} t}}\right)\right]
&\le e^{\beta_{\lambda}t+H(\xi)},
\end{align*}
which is \eqref{exp-ut0}.
Using the relation between the energy functional $H(u)$ and $\|\nabla u\|^2$ (more details we refer to \cite[Proposition 3.1]{CHL16b} for the proof of \eqref{exp-u}), we obtain \eqref{exp-ut0}.
\\\qed

\begin{cor}\label{exp-ud}
There exist  constants $C$ and $\alpha$ depending on $\xi$, 
$Q^{\frac12}$ and $T$ such that
\begin{align}\label{exp-ud0}
\sup_{t\in[0,T]}\E \left[\exp\left(\frac{\|\nabla u_\tau^D(t)\|^2}{e^{\alpha t}}\right)\right]
\le C
\end{align}
\end{cor}

\textbf{Proof}
The estimate \eqref{exp-ud0} follows immediately from the proof of Lemma \ref{exp-ut}.
\qed

\subsection{Strong Convergence Rate of Splitting Process}
\label{sec-spl-str}

Based on the $\HH^2$-a priori estimate in Proposition \ref{h2-sta} and the 
$\HH^1$-exponential integrability in Lemma \ref{exp-ut}, we can estimate the strong error between $u_\tau$ and $u$.

\begin{tm}\label{u-ut}
For any $p\ge 1$, there exist a constant $C=C(\xi,Q,T,p)$ such that
\begin{align}\label{u-ut0}
\left(\E \left[\sup_{m\in \Z_{M+1}}\|u(t_m)-u_\tau(t_m)\|^p\right]\right)^\frac1p
\le C \tau^\frac12.
\end{align}
\end{tm}

Denote by $e_{m+1}:=u(t_{m+1})-u_\tau(t_{m+1})$ in $T_m$, $m\in \Z_M$, the local error.
Note that $u_\tau(t_{m+1})=u_{\tau,m}^S(t_{m+1})$, 
$u_{\tau,m}^S(t_m)=u_{\tau,m}^D(t_{m+1})=u_\tau(t_m)+\int_{t_m}^{t_{m+1}}
\bi \Delta u_{\tau,m}^D(r)+\bi \lambda |u_{\tau,m}^D(r)|^2 u_{\tau,m}^D(r) dr$
and $e_m=u(t_m)-u_\tau^D(t_m)$.
Then
\begin{align}\label{umsm}
u(t_m)-u_{\tau,m}^S(t_m)=e_m-\bi \int_{t_m}^{t_{m+1}}
\left[\Delta u_{\tau,m}^D+\lambda |u_{\tau,m}^D|^2 u_{\tau,m}^D \right] dr.
\end{align}

We need the following representations of the differences $u-u_{\tau,m}^D$ and $u-u_{\tau,m}^S$ in $T_m$. 

\begin{lm}\label{u-ud-us}
For any $s\in T_m$ with $m\in \Z_M$, we have
\begin{align}\label{u-ud}
 u(s)-u_{\tau,m}^D(s) 
&=e_m+\int_{t_m}^s L_m^D dr-\bi\int_{t_m}^s u dW(r) 
\end{align}
and
\begin{align}\label{u-us}
 u(s)-u_{\tau,m}^S(s)  
&=e_m+\int_{t_m}^s L_m^S dr
-\bi\int_{t_m}^{t_{m+1}} \left[\Delta u^D_{\tau,m}+\lambda |u_{\tau,m}^D|^2u_{\tau,m}^D\right] dr \nonumber \\
&\quad -\bi\int_{t_m}^s \left[u-u_{\tau,m}^S\right] dW(r),
\end{align}
where 
\begin{align}\label{ld}
L_m^D:&=\bi \Delta\left[u-u_{\tau,m}^D\right]
+\bi \lambda (|u|^2 u-|u_{\tau,m}^D|^2 u_{\tau,m}^D)-\frac12  u F_Q,\\
\label{ls}
L_m^S:&=\bi \Delta u
+\bi \lambda |u|^2 u 
-\frac12 (u-u_{\tau,m}^S) F_Q.
\end{align}
\end{lm}

\textbf{Proof}
Note that $u(s)-u_{\tau,m}^D(s)=[u(s)-u(t_m)]+e_m+[u_{\tau,m}^D(t_m)-u_{\tau,m}^D(s)]$.
Combining \eqref{nls} and \eqref{nls-d}, we get \eqref{u-ud}.
Similarly, $u(s)-u_{\tau,m}^S(s)=[u(s)-u(t_m)]+e_m+[u_{\tau,m}^D(t_m)-u_\tau^D(t_{m+1})]+[u_{\tau,m}^S(t_m)-u_{\tau,m}^S(s)]$, which shows \eqref{u-us} by 
\eqref{nls}, \eqref{nls-d} and \eqref{nls-s}.
\qed

We also need to estimate the following stochastic integrals:
\begin{align*}
S^m_1:&=\int_{t_m}^{t_{m+1}}\left\|W(s)-W(t_m)\right\|_{\HH^1}^2ds, \\
S^m_2:&=\int_{t_m}^{t_{m+1}} 
\left\| \int_{t_m}^s u(r)dW(r)\right\|^2_{\HH^2} ds, \\
S^m_3:&=\int_{t_m}^{t_{m+1}}
\left\|\int_{t_m}^s  \int_{t_m}^r L_m^S dr_1dW(r)\right\| ds, \\
S^m_4:&=\int_{t_m}^{t_{m+1}}\left\|\int_{t_m}^s  \int_{t_m}^{r}
\left[u(r_1)-u^S(r_1)\right]dW(r_1)dW(r)\right\| ds.
\end{align*}

\begin{lm}\label{est-int} 
For any $p\ge 1$ and $m\in \Z_M$, there exists a constant $C=C(\xi,Q,T, p)$ such that 	
\begin{align*}
\|S^m_j\|_{\mathbb L^p(\Omega)}\le C\tau^2,\quad j=1,2,4; \quad
\|S^m_3\|_{\mathbb L^p(\Omega)}\le C\tau^{\frac 52}.
\end{align*}
\end{lm}

\textbf{Proof}
Let $p\ge 2$.
By Minkovski and Burkholder--Davis--Gundy inequalities and the a priori estimate \eqref{sta-u}, we have
\begin{align*}
\left\|S^m_2\right\|_{\mathbb L^p(\Omega)}
&\le \int_{t_m}^{t_{m+1}} 
\left\| \int_{t_m}^s u(r)dW(r)\right\|^2_{\mathbb L^{2p}(\Omega;\HH^2)} ds \\
&\le \|Q^{\frac12}\|^2_{\LL_2^2}\int_{t_m}^{t_{m+1}}\int_{t_m}^s 
\|u(r)\|^2_{\mathbb L^{2p}(\Omega;\HH^2)}drds
\le C\tau^2.
\end{align*}
This in turn shows $\left\|S^m_1\right\|_{\mathbb L^p(\Omega)}\le C\tau^2$ only by substituting $u\equiv 1$.

Applying Burkholder--Davis--Gundy inequality twice and the charge conservation law, we obtain 
\begin{align*}
&\|S^m_4\|_{\mathbb L^p(\Omega)} \\
&\le C \int_{t_m}^{t_{m+1}} 
\left[\int_{t_m}^s \sum_{k\in \N}\left\| \int_{t_m}^r\left(u(r_1)-u_\tau^S(r_1)\right) dW(r_1) Q^{\frac12}e_k\right\|_{\mathbb L^\frac p2(\Omega; \mathbb L^2)}^2 dr\right]^{\frac12} ds
\\
&\le C\|Q^{\frac12}\|_{\LL_2^2} \int_{t_m}^{t_{m+1}} 
\left[\int_{t_m}^s \left\| \int_{t_m}^r \left(u(r_1)-u_\tau^S(r_1) \right) dW(r_1)\right\|_{\mathbb L^p(\Omega; \mathbb L^2)}^2 dr\right]^{\frac12} ds\\
&\le C \|Q^{\frac12}\|_{\LL_2^2}^2 \int_{t_m}^{t_{m+1}} 
\left[\int_{t_m}^s \int_{t_m}^r \left\|u(r_1)-u_\tau^S(r_1)\right\|_{\mathbb L^p(\Omega; \mathbb L^2)}^2  dr_1dr \right]^{\frac12} ds\le C \tau^2.
\end{align*}
Similar arguments yield that $\|S^m_3\|_{\mathbb L^p(\Omega)}\le C\tau^\frac52$ and $\|S^m_4\|_{\mathbb L^p(\Omega)}\le C\tau^\frac52$ for $p\ge 2$.
We complete the proof for $p\in [1,2)$ by H\"older inequality.
\qed

\begin{lm}\label{u-ut-exp}
For any $p\ge 1$, there exists a constant $C=C(\xi,Q,T)$ such that 
\begin{align}\label{u-ut-exp0}
\left\|\exp\left(2\int_0^T  \|u(s)\|_{\mathbb L^\infty}
\|u_\tau^D(s)\|_{\mathbb L^\infty} ds\right)\right\|_{\mathbb L^p(\Omega)}\le C
\end{align}
\end{lm}

\textbf{Proof}
By Cauchy--Schwarz, Gagliardo--Nirenberg, Young, Jensen and Minkovski inequalities, we get
\begin{align*}
&\left\|\exp\left(2\int_0^T  \|u(s)\|_{\mathbb L^\infty}
\|u_\tau(s)\|_{\mathbb L^\infty} ds\right)\right\|_{\mathbb L^p(\Omega)} \\
&\le \left\|\exp\left(\int_0^T  2\|\xi\|  \|\nabla u\| ds\right)\right\|_{\mathbb L^{2p}(\Omega)}
\left\|\exp\left(\int_0^T 2\|\xi\|  \|\nabla u^D_\tau\| ds\right)\right\|_{\mathbb L^{2p}(\Omega)} \\
&\le \exp\left(\int_0^T 2pTe^{\alpha T} \|\xi\|^2 ds\right)
\left\|\exp\left( \int_0^T \frac{\|\nabla u\|^2}{2pTe^{\alpha T}} ds\right)\right\|_{\mathbb L^{2p}(\Omega)} \\
&\quad \times
\exp\left(\int_0^T 2pTe^{\alpha T}\|\xi\|^2 ds\right)
\left\|\exp\left( \int_0^T \frac{\|\nabla u_\tau^D\|^2}{2pTe^{\alpha T}} ds\right)\right\|_{\mathbb L^{2p}(\Omega)}  \\
&\le  C \left(\sup_{t\in[0,T]}\E \left[\exp\left(\frac {\|\nabla u(t)\|^2}{e^{\alpha t}}\right)\right]
\cdot \sup_{t\in[0,T]}\E \left[\exp\left(\frac{\|\nabla u_\tau^D(t)\|^2}{e^{\alpha t}}\right)\right]\right)^{\frac 1{2p}},
\end{align*}\\
where $\alpha$ is presented in \eqref{exp-u} or \eqref{exp-ut00}.
From the above estimations we obtain \eqref{u-ut-exp0} combining Theorem \ref{well} and Corollary \ref{exp-ud}.
\qed

Our last preliminary result is the following discrete Gronwall inequality(see \cite[Lemma 1.4.2]{QV94}).

\begin{lm} \label{dis-gro}
Let $m\in \N$ and $\{p_j\}_{j\in \N}$, $\{k_j\}_{j\in \N}$ are nonnegative number sequences.
Assume that the sequence $\{\phi_j\}_{j\in \N}$ satisfies 
\begin{align*}
\phi_{m+1}&\le \phi_0+\sum_{j=0}^m p_j+\sum_{j=0}^m k_j\phi_j,
\end{align*}
then 
\begin{align*}
\phi_{m+1}\le \left(\phi_0+\sum_{j=0}^m p_j\right)
\exp\left(\sum_{j=0}^m k_j\right).
\end{align*}	
\end{lm}
\textbf{Proof of Theorem \ref{u-ut}}
\\
By It\^o formula and \eqref{umsm}, we have
\begin{align*}
\|e_{m+1}\|^2
&=\left\|e_m-\int_{t_m}^{t_{m+1}}
\bi \left[\Delta u_{\tau,m}^D+\lambda |u_{\tau,m}^D|^2 u_{\tau,m}^D\right] dr\right\|^2  \\
&\quad +2\int_{t_m}^{t_{m+1}} \left\<u-u_{\tau,m}^S,
\bi \left[\Delta u+\lambda |u|^2 u \right] \right\>ds
:=I^m_1+I^m_2.
\end{align*}
The first item $I^m_1$ has the estimation by Gagliardo--Nirenberg inequality:
\begin{align*}
I^m_1
&= \|e_m\|^2 -2\left\< e_m,\int_{t_m}^{t_{m+1}}  
\bi \left[\Delta u_{\tau,m}^D+\lambda |u_{\tau,m}^D|^2 u_{\tau,m}^D\right] ds \right\>\\
&\quad+\left\|\int_{t_m}^{t_{m+1}} \left[\Delta u_{\tau,m}^D+\lambda |u_{\tau,m}^D|^2 u_{\tau,m}^D\right] ds\right\|^2\\
&\le \|e_m\|^2 -2\left\< e_m,\int_{t_m}^{t_{m+1}}  
\bi \left[\Delta u_{\tau,m}^D+\lambda |u_{\tau,m}^D|^2 u_\tau^D\right] ds 
\right\>\\
&\quad+C\tau^2 \left(1+\sup_{t\in T_m}\|u_{\tau,m}^D(t)\|_{\HH^2}^2\right).
\end{align*}	
Substituting \eqref{u-us} into $u-u_{\tau,m}^S$, we divide $I^m_2$ into
\begin{align*}
I_2^m
&=2\int_{t_m}^{t_{m+1}} \Big\<\int_{t_m}^s L_m^S(r)dr , \bi(\Delta u(s)+ \lambda |u(s)|^2u(s)) \Big\> ds\\
&\quad +2\int_{t_m}^{t_{m+1}} \left\<e_m,\bi (\Delta u(s)+\lambda |u(s)|^2u(s) )\right\> ds\\
&\quad- 2\bigg\<\int_{t_m}^{t_{m+1}} \Delta u^D_{\tau,m}(r)+ \lambda |u^D_{\tau,m}(r)|^2u_{\tau,m}^D(r)dr, \\
&\qquad \qquad \int_{t_m}^{t_{m+1}}  \Delta u(s)+ \lambda |u(s)|^2u(s) ds \bigg\> \\
&\quad-2\int_{t_m}^{t_{m+1}} \Big\<\int_{t_m}^s  (u(r)-u_{\tau,m}^S(r))dW(r),
\Delta u(s)+ \lambda |u(s)|^2u(s) \Big\> ds\\
&:=I_{21}^m+I_{22}^m+I_{23}^m+I_{24}^m.
\end{align*}
By H\"older and Gagliardo--Nirenberg inequalities, we get
\begin{align*}
I^m_{21}
&\le C\tau^2 \left(1+\sup_{t\in T_m} \|u(t)\|_{\HH^2}^2+\sup_{t\in T_m} \|u_\tau(t)\|_{\HH^2}^2\right),\\
I^m_{23}
&\le C\tau^2 \left(1+\sup_{t\in T_m} \|u(t)\|_{\HH^2}^2+\sup_{t\in T_m} \|u_{\tau,m}^D(t)\|_{\HH^2}^2\right).
\end{align*}
Substituting  \eqref{u-us} into $I^m_{24}$, we have
\begin{align*} 
I^m_{24}
&=-2\int_{t_m}^{t_{m+1}}\left\<  e_m \left[W(s)-W(t_m)\right],
\Delta u+ \lambda |u|^2 u \right\> ds\\
&\quad-2\int_{t_m}^{t_{m+1}}\Bigg\< \int_{t_m}^s  
\int_{t_m}^r L_m^S(u(r_1)) dr_1dW(r),
\Delta u+ \lambda |u|^2 u \Bigg\> ds \\
&\quad+2\Bigg\<\bi \int_{t_m}^{t_{m+1}}  
\left[\Delta u_{\tau,m}^D+\lambda |u_{\tau,m}^D|^2u_{\tau,m}^D\right] dr_1, \\
&\qquad \qquad \int_{t_m}^{t_{m+1}} \left[\Delta u+ \lambda |u|^2 u\right]
[W(s)-W(t_m)]  ds \Bigg\>  \\
&\quad +2\int_{t_m}^{t_{m+1}}\Big\< \int_{t_m}^s  \int_{t_m}^{r}
\bi\left[u-u_{\tau,m}^S\right] dW(r_1)dW(r),
\Delta u+ \lambda |u|^2 u \Big\> ds   \\
&:=I^m_{241}+I^m_{242}+I^m_{243}+I^m_{244}.
\end{align*}  
Cauchy-Schwarz and Gagliardo--Nirenberg inequalities and the charge conservation law for Eq. \eqref{nls} yield that 
\begin{align*}
I^m_{241}
&\le \tau\|e_m\|^2+C\left(1+\sup_{t\in T_m}\|u(t)\|_{\HH^2}^2\right) S^m_1,
\end{align*} 
For other terms, we have
\begin{align*}
I^m_{243}
&\le C\tau^2 \left(1+\sup_{t\in T_m}\|u(t)\|_{\HH^2}^2\right)
+C\tau \left(1+\sup_{t\in T_m}\|u(t)\|_{\HH^2}^2 \right) S^m_1.
\end{align*}

For other terms, we have
\begin{align*}
I^m_{242}+I^m_{244}
&\le C\left(1+\sup_{t\in T_m}\|u(t)\|_{\HH^2}\right) (S^m_3+S^m_4).
\end{align*}
Then
\begin{align*}
I^m_{24}
&\le \tau\|e_m\|^2
+C\left(1+\sup_{t\in T_m}\|u(t)\|_{\HH^2}^2\right) (\tau^2+S^m_1+S^m_3+S^m_4).
\end{align*}	
Summing up $I^m_1$ and $I^m_2$ and integrating by parts, we deduce that 
\begin{align*}
&\|e_{m+1}\|^2-\|e_m\|^2 \\
&\le \tau\|e_m\|^2+2\int_{t_m}^{t_{m+1}}  
\left\<\Delta e_m, \bi \left[u-u_{\tau,m}^D \right] \right\>  ds  \\
&\quad +2\lambda \int_{t_m}^{t_{m+1}}  \left\< e_m, 
\bi \left[|u|^2 u-|u_{\tau,m}^D|^2 u_{\tau,m}^D \right]\right\>  ds  \\
&\quad +C\tau^2 \bigg(1+\sup_{t\in T_m} 
\|u_{\tau,m}^D(t)\|_{\HH^2}^2+\sup_{t\in T_m} 
\|u_{\tau,m}(t)\|_{\HH^2}^2
+\sup_{t\in T_m} \|u(t)\|_{\HH^2}^2\bigg) \\
&\quad +C\left(1+\sup_{t\in T_m}\|u(t)\|_{\HH^2}^2\right) (S^m_1+S^m_3+S^m_4).
\end{align*}
Denote by 
\begin{align*}
J^m_1:&=2\int_{t_m}^{t_{m+1}}  
\left\<\Delta e_m, \bi \left[u-u_{\tau,m}^D \right] \right\>  ds, \\
J^m_2:&=2\lambda \int_{t_m}^{t_{m+1}}  \left\< e_m, 
\bi \left[|u|^2 u-|u_{\tau,m}^D|^2 u_{\tau,m}^D \right]\right\>  ds.
\end{align*}
	%
	%
	Substituting \eqref{u-ud} into $J^m_1$, by H\"older inequality
	and integration by parts,
	we have 
\begin{align*}
J^m_1
&=2\int_{t_m}^{t_{m+1}}  \left\<\Delta e_m, 
\bi \int_{t_m}^s L_m^D(r) dr\right\> ds\\
&\quad +2\int_{t_m}^{t_{m+1}} 
\left\<e_m, \Delta \left(\int_{t_m}^s u(r)dW(r)\right) \right\> ds\\
&\le \tau \|e_m\|^2+C\tau^2 \left(1+\sup_{t\in T_m} \|u(t)\|_{\HH^2}^2+\sup_{t\in T_m}\|u_{\tau,m}^D(t)\|_{\HH^2}^2 \right)+CS^m_2.
\end{align*}
	For term $J^m_2$, using cubic difference formula $|a|^2a-|b|^2b=(|a|^2+|b|^2)(a-b)+ab(\overline a-\overline b)$ and \eqref{u-ud} in Lemma \ref{u-ud-us}, we obtain 
\begin{align*}
J^m_2
&=2\lambda \int_{t_m}^{t_{m+1}} 
\Bigg\< e_m, \bi \left(|u|^2 +|u_{\tau,m}^D|^2\right) 
\left[\int_{t_m}^s \mathbb L^D(r) dr -\bi \int_{t_m}^s  u(r) dW(r) \right] \Bigg\> ds\\
&\quad+2\lambda \int_{t_m}^{t_{m+1}} 
\Bigg\< e_m, \bi u u_{\tau,m}^D \left[\overline e_m+\int_{t_m}^s \overline{L_m^D}(r) dr +\bi \int_{t_m}^s \overline u(r) dW(r)\right] \Bigg\> ds.
\end{align*}
Cauchy-Schwarz and Gagliardo--Nirenberg inequalities imply
\begin{align*}
J^m_2
&\le \left(2\tau+2\int_{t_m}^{t_{m+1}} 
\|u\|_{L_{\infty}}\|u_{\tau,m}^D\|_{L_{\infty}} ds\right) \|e_m\|^2
+C \tau^3\Big(\sup_{t\in T_m}\|u(t)\|_{\mathbb L^\infty}^4\\
&\quad +\sup_{t\in T_m}\|u_{\tau,m}^D(t)\|_{\mathbb L^\infty}^4 \Big) \Big(1+\sup_{t\in T_m} \|u(t)\|_{\HH^2}^2
+\sup_{t\in T_m}\|u_{\tau,m}^D(t)\|_{\HH^2}^2 + S^m_2\Big).
\end{align*}
Therefore, we obtain
\begin{align*}
\|e_{m+1}\|^2
&\le \|e_m\|^2+\left(4 \tau+2\int_{t_m}^{t_{m+1}} \|u(s)\|_{L_{\infty}}\|u_{\tau,m}^D(s)\|_{L_{\infty}}ds \right) \|e_m\|^2   \\
&\quad+C\left(1+\sup_{t\in T_m}\|u(t)\|_{\HH^2}^6
+\sup_{t\in T_m}\|u_{\tau,m}(t)\|_{\HH^2}^6+
\sup_{t\in T_m}\|u_{\tau,m}^D(t)\|_{\HH^2}^6\right) \\ 
&\quad \times \left(\tau^2+\sum_{j=1}^4 S^m_j\right).
\end{align*}
Set 
\begin{align*}
k_m:&=4 \tau+2\int_{t_m}^{t_{m+1}} \|u(s)\|_{L_{\infty}}\|u_{\tau,m}^D(s)\|_{L_{\infty}}ds, \\
p_m:&=C\left(1+\sup_{t\in T_m}\|u(t)\|_{\HH^2}^6
+\sup_{t\in T_m}\|u_{\tau,m}(t)\|_{\HH^2}^6
+\sup_{t\in T_m}\|u_{\tau,m}^D(t)\|_{\HH^2}^6\right) \\
&\quad \times \left(\tau^2+\sum_{j=1}^4 S^m_j\right).
\end{align*}
Then
\begin{align*}
\|e_{m+1}\|^2
\le \|e_m\|^2+k_m \|e_m\|^2+p_m\le \cdots
\le \|e_0\|^2+\sum_{n=0}^m k_n \|e_n\|^2+\sum_{n=0}^m p_n. 
\end{align*}
Applying Lemma \ref{dis-gro}, we have
\begin{align*}
&\|e_{m+1}\|^2 \\ 
&\le C\exp\left(4T+2 \int_0^{t_{m+1}} \|u(t)\|_{L_{\infty}} \|u_{\tau_m}^D(t)\|_{L_{\infty}} dt\right) \times \left(\tau+\sum_{n=0}^m \left[\sum_{j=1}^4 S^n_j\right]\right) \\
&\qquad \times \left(1+\sup_{t\in [0,T]}\|u(t)\|_{\HH^2}^6
+\sup_{t\in [0,T]}\|u_\tau(t)\|_{\HH^2}^6
+\sup_{t\in [0,T]}\|u_\tau^D(t)\|_{\HH^2}^6\right).
\end{align*}
Then taking $\frac p2$-moments on both sides and using H\"older inequality, we obtain
\begin{align*}
&\E \left[\sup_{m\in \Z_{M+1}}\|e_{m+1}\|^p\right]  \\
&\le C\Bigg\|\exp\left(2\int_0^T \|u(t)\|_{L_{\infty}} \|u_\tau^D(t)\|_{L_{\infty}} dt\right)\Bigg\|_{\mathbb L^p(\Omega)}^\frac p2
\left\|\tau+\sum_{n=0}^M \sum_{j=1}^4 S^n_j \right\|_{\mathbb L^{2p}(\Omega)}^\frac p2   \\
&\quad \times \left(1+\left(\E\left[\sup_{t\in [0,T]}\|u(t)\|_{\HH^2}^{12p}\right]\right)^{\frac 14}
+\left(\E\left[\sup_{t\in [0,T]}\|u_\tau^D(t)\|_{\HH^2}^{12p}\right] \right)^{\frac 14}\right). 
\end{align*}
Now the exponential moments' estimation \eqref{u-ut-exp0} in Lemma \ref{u-ut-exp} and $\HH^2$-a priori estimations  \eqref{sta-u} and \eqref{h2-sta0}  imply that 
\begin{align*}
\E \left[\sup_{m\in \Z_{M+1}}\|e_{m+1}\|^p\right] 
\le \left(\tau+\sum_{n=0}^M \sum_{j=1}^4 \left\|S^n_j \right\|_{\mathbb L^{2p}(\Omega)}\right)^\frac p2.
\end{align*}
We complete the proof of \eqref{u-ut0} by Lemma \ref{est-int}.
\\\qed

\begin{rk}
From the proof, we can obtain a continuous version of the estimation \eqref{u-ut0}:
\begin{align*}
\left(\E \left[\sup_{t\in [0,T]}\|u(t)-u_\tau(t)\|^p\right]\right)^\frac1p
\le C \tau^\frac12.
\end{align*}
\end{rk}

\section{Splitting Crank--Nicolson Scheme}
\label{sec-cn}

For the nonlinear case,  the splitting process $u_\tau$  defined by \eqref{spl} is not a proper implementary  numerical method since Eq. \eqref{nls-d} does not possess a analytic solution. 
To obtain a temporal discretization,
we use the Crank--Nicolson scheme to temporally discretize  $u_\tau=\{u_\tau(t):\ t\in [0,T]\}$ and get the following splitting Crank--Nicolson scheme starting from $\xi$:
\begin{align}\label{spl-cn}
\begin{cases}
u^D_{m+1}
= u_m+ \bi \tau \Delta  u^D_{m+\frac12}+\bi \lambda \tau \frac {| u_m|^2+| u_{m+1}^D|^2}2  u^D_{m+\frac12},\\
u_{m+1}=\exp\left(-\bi (W_{t_{m+1}}-W_{t_m})\right) u_{m+1}^D,
\quad m\in \Z_{M-1},
\end{cases}
\end{align}
where $ u_{m+\frac12}^D=\frac12(u_m+ u_{m+1}^D)$.
It is not difficult to show that $\| u_m\|=\| u^D_{m+1}\|$ and $H(u_m)=H(u^D_{m+1})$, $m\in \Z_{M-1}$.
Moreover, the splitting Crank--Nicolson scheme \eqref{spl-cn} preserves the discrete charge a.s., i.e., 
\begin{align}\label{cn-cha}
\| u_m\|^2=\|\xi\|^2,\quad m\in \Z_{M}.
\end{align} 
Similarly to $u_\tau$, $u_m$ has a local continuous extension $u_m^S$ in $T_m$, $m\in \Z_{M-1}$, through the stochastic flow of Eq. \eqref{nls-s}.
 
Our main goal in this section is to estimate the strong convergence rate of the splitting Crank--Nicolson scheme \eqref{spl-cn}.
At first, we show the exponential integrability of $\{u_m\}_{m\in \Z_{M}}$, which implies the $\HH^1$-a priori estimate. 

\begin{lm}\label{exp-um}
There exist  constants $C$ and $\alpha$ depending on $\xi$, $Q$ and $T$ such that 
\begin{align}\label{exp-um0}
\sup_{m\in\Z_{M}}\E \left[\exp\left(\frac {H( u_m)}{e^{\alpha t_m}}\right)\right]
\le C
\end{align}
and 
\begin{align}\label{exp-um1}
\sup_{m\in\Z_{M}}\E \left[\exp\left(\frac {\|\nabla u_m\|^2}{e^{\alpha t_m}}\right)\right]
\le C.
\end{align}
\end{lm}

\textbf{Proof}
The proof is similar to that of Lemma \ref{exp-ut} and we omit the details.
\qed
	
\begin{cor}\label{h1-um}
For any $p\ge 1$, there exist a constant $C=C(\xi,Q,T,p)$ such that 
\begin{align}\label{h1-um0}
\E \left[\sup_{m\in\Z_{M}}\|u_m\|_{\HH^1}^p \right]\le C.
\end{align}
\end{cor}

\textbf{Proof}
By Lemma \ref{exp-um}, we have 
\begin{align*}
\sup_{m\in\Z_{M}}\E \left[\|u_m\|_{\HH^1}^p \right]\le C.
\end{align*}
The above inequality in turn shows \eqref{h1-um0} by similar arguments in Lemma \ref{h1-sta}.
\qed	

Our next technical requirement is the uniform $\HH^2$-a priori estimate of $u_m$, $m\in \Z_{M}$. 
We need the following useful result.
For convenience, we recall that
\begin{align}\label{cn}
 u^{D}_{m+1}-u_{m}
=\bi \tau\Delta  u^{D}_{m+\frac12}  + \bi \lambda \tau  \frac {|u_m|^2+|u^{D}_{m+1}|^2}2 u^{D}_{m+\frac12}, \quad m\in \Z_M.	
\end{align} 

\begin{lm}\label{discon}
Assume that $ u_m, u^D_{m+1}\in \HH^1_0\cap \HH^2$ for some $m\in \Z_M$. 
There exists a constant $C=C(\xi,Q)$ such that 
\begin{align*}
\| u^D_{m+1}- u_m\|^2
&\le C\tau\left(\|\nabla u^D_{m+1} \|^2+\|\nabla u_m \|^2\right),\\
\| u^D_{m+1}- u_m\|
&\le C\tau (\|\nabla  u^D_{m+1}\| +\|\nabla  u_m\|)
+\frac\tau2(\|\Delta  u_m\|+\|\Delta  u^D_{m+1}\|), \\
\|\nabla  u^D_{m+1}-\nabla  u_m\|^2
&\le C\tau \left(\|\nabla u^D_{m+1}\|^4+\|\nabla u_m\|^4\right)
+\frac{\tau}2(\|\Delta  u_m\|^2+\|\Delta  u_{m+1}\|^2).
\end{align*}
\end{lm}
	
\textbf{Proof}
Taking inner product with $ u^D_{m+1}- u_m$ on \eqref{cn} and using Gagliardo--Nirenberg inequality, we obtain 
\begin{align*}
\| u^D_{m+1}- u_m\|^2
&=\tau\left\< u^D_{m+1}- u_m, \bi \Delta  u^D_{m+\frac12}\right\>\\
&\quad +\lambda \tau \left\< u^D_{m+1}- u_m, \bi \frac {| u^D_{m+1}|^2+| u_m|^2}2 u^D_{m+\frac12} \right\> \\
&\le C\tau\left(\|\nabla u^D_{m+1} \|^2+\|\nabla u_m \|^2\right).
\end{align*}
Taking $\mathbb L^2$-norm on both sides of \eqref{cn}, we  get
\begin{align*}
\| u^D_{m+1}- u_m\|
&\le \frac\tau2(\|\Delta  u_m\|+\|\Delta  u^D_{m+1}\|)+C\tau (\|\nabla  u^D_{m+1}\| +\|\nabla  u_m\|).
\end{align*}
Taking inner product with $\Delta(u^D_{m+1}-u_m)$ on \eqref{cn}, integrating by parts and using Gagliardo--Nirenberg inequality, we obtain 
\begin{align*}
&\|\nabla u^D_{m+1}-\nabla u_m \|^2\\
&=\left\<\Delta  u^D_{m+1}-\Delta u_m,\bi \tau \Delta  u^D_{m+\frac12} 
+\bi \tau \frac {|  u^D_{m+1}|^2+| u_m|^2}2
 u^D_{m+\frac12} \right\>\\
&\le \frac{\tau }2\left(\|\Delta  u^D_{m+1}\|^2 +\|\Delta  u_m\|^2\right)
+C\tau \left(\|\nabla u^D_{m+1}\|^4+\|\nabla u_m\|^4\right).
\end{align*}
This complete the proof.
\\\qed

\begin{lm}\label{h2-um}
For any $p\ge 1$, there exists a constant $C=C(\xi,Q, T, p)$ such that  
	\begin{align}\label{h2-um0}
	\E\left[\sup_{m\in \Z_{M}}\|u_m\|_{\HH^2}^p \right]\le C.
	\end{align}
\end{lm}

\textbf{Proof}
In terms of the $\HH^1$-a priori estimate \eqref{h1-um0} in Corollary \ref{h1-um}, we focus on the a priori estimate in $\HH^2$.
	Let $m\in \Z_M$.
	Taking  complex inner product on both sides of \eqref{cn} with $\Delta (u^{D}_{m+1}-u_m)$,
	and then taking the imaginary part, we obtain 
	\begin{align*}
	& \|\Delta u^{D}_{m+1}\|^2-\|\Delta  u_m\|^2 \\
	&=-\frac { \lambda}2\left\<\Delta (u^{D}_{m+1} - u_m),   (| u_m|^2+|u^{D}_{m+1}|^2) (u_{m}+u^D_{m+1})   \right\>.
	\end{align*}
	 By the identity $(|a|^2+|b|^2)(a+b)=2|a|^2a+2|b|^2b-(|b|^2-|a|^2)(b-a)$, $a,b\in \mathbb C$ we have 
	\begin{align*}
	&\left\<\Delta (u^{D}_{m+1} - u_m),   (| u_m|^2+|u^{D}_{m+1}|^2) (u_m+u^D_{m+1})   \right\>\\
	&=2\left\<\Delta u^{D}_{m+1},  |u^{D}_{m+1}|^2u^{D}_{m+1}\right\>
	-
	2\left\<\Delta   u_m,   |u_m|^2 u_m\right\>\\
	&\quad+\frac12\left\<\Delta u^{D}_{m+1}-\Delta   u_m,   \left(|u^{D}_{m+1}|^2-|u_m|^2\right)\left(  u_m-u^{D}_{m+1}\right)\right\>\\
	&\quad -\left\<\Delta u^{D}_{m+1} , \left(|u^{D}_{m+1}|^2u^{D}_{m+1}-| u_m|^2u_m\right) \right\>\\
	&\quad + \left\<\Delta u^{D}_{m+1}-\Delta u_m , | u^{D}_{m+1}|^2 u^{D}_{m+1}\right\>.
	\end{align*}
Recalling the definition \eqref{h2-lya} of the Lyapunov functional $f$, we have  
	\begin{align*}
&f( u^{D}_{m+1})- f( u_m)\\
&=\|\Delta u^{D}_{m+1}\|^2-\|\Delta  u_m\|^2  \\
&\quad +\lambda \left[\<\Delta u^{D}_{m+1}, |u^{D}_{m+1}|^2 u^{D}_{m+1}\> 
-\<\Delta u_m, |u_m|^2 u_m\>\right]  \\
&=\lambda  \left\<\Delta u^{D}_{m+1} , \left(|u^{D}_{m+1}|^2u^{D}_{m+1}-|u_m|^2u_m\right) \right\>  \\ 
&\quad -\lambda  \left\<\Delta(u^{D}_{m+1}-u_m),
 | u^{D}_{m+1}|^2 u^{D}_{m+1}\right\>\\
&\quad +\frac\lambda 2\left\<\Delta(u^{D}_{m+1}-u_m), \left(|u^{D}_{m+1}|^2-|u_m|^2\right)\left(  u_m-u^{D}_{m+1}\right)\right\>.
	\end{align*}
	By cubic difference formula, we 
	have 
	\begin{align*}
	\small
	&\left\<\Delta u^{D}_{m+1} ,  \left(|u^{D}_{m+1}|^2u^{D}_{m+1}-|u_m|^2u_m\right) \right\> \\
	&= \left\<\Delta u^{D}_{m+1} , \left(|u^{D}_{m+1}|^2+| u_m|^2\right)\left(u^{D}_{m+1}-  u_m\right)\right\>  \\
	&\quad + \left\<\Delta u^{D}_{m+1} , u^{D}_{m+1} u_m\left(\overline {u^{D}_{m+1}}-\overline {  u_m}\right)\right\>\\
	&=  \left\<\Delta u^{D}_{m+1} ,   \left(|u^{D}_{m+1}|^2+| u_m|^2\right)\left(u^{D}_{m+1}- u_m\right)\right\> \\
	&\quad -\left\<\Delta u^{D}_{m+1} , u^{D}_{m+1}|u^{D}_{m+1}- u_m|^2\right\>\\
	&\quad  +\left\<\Delta u^{D}_{m+1} ,   (u^{D}_{m+1})^2 \left(\overline {u^{D}_{m+1}}-\overline { u_m}\right)\right\>.
	\end{align*}
	On the other hand, integration by parts and using the equality $\Delta(|u|^2u)=2|u|^2 \Delta u+4|\nabla u|^2u+2(\nabla u)^2\overline u+(u)^2\Delta \overline u$ yield that
	\begin{align*}
	&\left\<\Delta u^{D}_{m+1}-\Delta  u_m , | u^{D}_{m+1}|^2 u^{D}_{m+1}\right\>\\
	&=\left\<u^{D}_{m+1}- u_m , \left(2| u^{D}_{m+1}|^2 \Delta u^{D}_{m+1} + \Delta \overline{u^{D}_{m+1}} (u^{D}_{m+1})^2\right) \right\>\\
	&\quad +\left\<u^{D}_{m+1}-  u_m ,  \left( 4|\nabla u^{D}_{m+1}|^2  u^{D}_{m+1} + 2(\nabla \overline{u^{D}_{m+1}})^2 u^{D}_{m+1}\right) \right\>.
	\end{align*}
Therefore, 
	\begin{align*}
	&f( u^{D}_{m+1})- f(u_m)\\
	&= \lambda \left\<\Delta u^{D}_{m+1} , \left(-|u^{D}_{m+1}|^2+|u_m|^2\right)\left(u^{D}_{m+1}-  u_m\right)\right\>\\
	&\quad -\lambda \left\<\Delta u^{D}_{m+1} ,  u^{D}_{m+1}|u^{D}_{m+1}- u_m|^2\right\>\\
	&\quad -\lambda \left\<u^{D}_{m+1}- u_m ,
	 4\left(|\nabla u^{D}_{m+1}|^2  u^{D}_{m+1}+2(\nabla \overline{u^{D}_{m+1}})^2 u^{D}_{m+1}\right) \right\> \\
	&\quad -\frac \lambda  2\left\<\Delta (u^{D}_{m+1}-u_m),    \left(|u^{D}_{m+1}|^2-|u_m|^2\right)\left(  u_m-u^{D}_{m+1}\right)\right\>  \\
	&:=II^m_1+II^m_2+II_3^m+II_4^m.
	\end{align*}
By H\"older, Young and Gagliardo--Nirenberg inequalities,  we obtain 
	\begin{align*}
	II^m_1&\le \|u^{D}_{m+1}- u_m \|_{L_4}^2\|\Delta u^{D}_{m+1}\|(\|u^{D}_{m+1}\|_{L_{\infty}}+\|u_m\|_{L_{\infty}})\\
	&\le  \frac \tau {16} \|\Delta u^{D}_{m+1}\|^2+\frac C\tau  \|\nabla u^{D}_{m+1}- \nabla  u_m \|
	\|u^{D}_{m+1}-   u_m \|^3 \\
	&\quad \times (\|u^{D}_{m+1}\|^2_{L_{\infty}}+\|u_m\|^2_{L_{\infty}})\\
	&\le  \frac \tau {16}\|\Delta u^{D}_{m+1}\|^2+\frac 14 \|\nabla u^{D}_{m+1}- \nabla u_m \|^2 \\
&\quad +C\tau \left(\|\nabla u^{D}_{m+1}\|^4+\|\nabla u_m\|^4\right)
+\frac 1 {\tau^5} \|u^{D}_{m+1}-  u_m \|^{12}.
	\end{align*}
	Similarly, we get 
	\begin{align*}
	II^m_2+II_4^m
	&\le  \frac \tau{16}(\|\Delta u^{D}_{m+1}\|^2+ \|\Delta u_m\|^2)+\frac 14 \|\nabla u^{D}_{m+1}- \nabla u_m \|^2\\
	&\quad+C\tau \left(\|\nabla u^{D}_{m+1}\|^4+\|\nabla u_{m}\|^4\right)
	+\frac 1 {\tau^5} \|u^{D}_{m+1}- u_m \|^{12}.
	\end{align*}
	Lemma \ref{discon} and the charge and  energy conservation laws of \eqref{cn} yield that
	\begin{align*}
	II_3^m
	&\le  C\|u^{D}_{m+1}-  u_m \|\|\nabla u^{D}_{m+1}\|^2\|u^{D}_{m+1}\|_{L_{\infty}}\\
	&\le C\tau\|\nabla u^{D}_{m+1}\|^{\frac 52}\left(\|\Delta u^{D}_{m+\frac12}\|+\|\nabla u^{D}_{m+1} \|+\|\nabla u_m\|\right)\\
	&\le \frac \tau 2 \|\Delta u^{D}_{m+\frac12}\|^2+C\tau (\|\nabla u^{D}_{m+1}\|^5+\|\nabla u^{D}_{m+1}\|^{\frac 72}+ \|\nabla u^{D}_{m+1}\|^{\frac 52}\|\nabla  u_m\|).
	\end{align*}
	 Again by Lemma \ref{discon}, we conclude that 
	\begin{align*}
	f( u^{D}_{m+1})
	&\le  f(  u_m)
	+ \frac12 \tau (\|\Delta u_m\|^2 +\|\Delta u^{D}_{m+1}\|^2) \\
&\quad +C \tau \left(1+\|\nabla u^{D}_{m+1}\|^{12}+\|\nabla u_m\|^{12}\right).
	\end{align*}
By \eqref{h2-lya}, integrating by parts and Sobolev inequality, we achieve
\begin{align*}
f( u^{D}_{m+1})
&\le  \frac{1+\frac\tau 2}{1-\frac\tau 2}f(  u_m)
+\frac{C\tau}{1-\frac\tau 2}\left(1+\|\nabla u^{D}_{m+1}\|^{12}+\|\nabla u_m\|^{12}\right).
\end{align*}
The fact $\tau<1$ implies
		\begin{align*}
		f( u^{D}_{m+1})
		&\le (1+2\tau) f(  u_m)
		+C \tau \left(1+\|\nabla u^{D}_{m+1}\|^{12}+\|\nabla u_m\|^{12}\right).
		\end{align*}
Notice that $u_m$ can be extend to a continuous process $u^S_m(t)$ as the solution of Eq. \eqref{nls-s} with initial data $u^{D}_{m+1}$ in $ T_m$.
The same arguments to derive \eqref{uts} lead to
\begin{align*}
\E \left[f(u_m^S(t))\right]-\E \left[f(u^{D}_{m+1})\right]
\le C \tau+C\int_{t_m}^t \E\left[ f(u_m^S(r))\right] dr,
\quad t\in T_m.
\end{align*}
Gronwall inequality and the above two inequalities imply
\begin{align*}
\E \left[f(u_{m+1})\right]
&\le e^{C\tau} \left( (1+2\tau)\E \left[f( u_{m})\right]+C \tau \right).
\end{align*}
Then Lemma \ref{dis-gro} yields
\begin{align*}
\E \left[f(u_{m+1})\right]
&\le C e^{CT}\left(1+f(\xi)\right),
\end{align*}
which in turn shows 
\begin{align*}
\sup_{m\in \Z_{M}}\E \left[\|\Delta u_m \|^2\right]\le C.
\end{align*}
Similar arguments to derive \eqref{h2-sta0} in Proposition \ref{h2-sta} complete the proof of \eqref{h2-um0}.
\\\qed

\begin{cor}\label{h2-umd}
For any $p\ge1$, there exist a constant $C=C(\xi,Q,T,p)$ such that
\begin{align}\label{h2-umd0}
	\E\left[\sup_{m\in \Z_{M}\setminus \{0\} }\|u_m^D\|_{\HH^2}^p \right]\le C.
	\end{align}
\end{cor}

\textbf{Proof}
The proof of \eqref{h2-umd0} follows immediately from the proof of \eqref{h2-um0}.
\qed

Based on Lemmas \ref{exp-um} and \ref{h2-um}, in the rest of this section we prove that the splitting Crank--Nicolson scheme \eqref{spl-cn} possesses strong convergence rate $\OO(\tau^{\frac12})$.
To the best of our knowledge, this is the first result about strong convergence rate of temporal discretization for SPDEs with non-monotone coefficients.

We begin with the following lemma which is a discrete version of Lemma \ref{u-ut-exp}.

\begin{lm}\label{uutexp}
For any $p\ge 1$, there exist a constant $C=C(\xi,Q,T,p)$ such that
\begin{align*}
\left\|\exp\left(2\tau\sum_{m\in\Z_{M}}\left\|u_\tau(t_m)\right\|_{L_{\infty}}\left\|u_m\right\|_{L_\infty}\right)\right\|_{\mathbb L^p(\Omega)}
\le C.
\end{align*}
\end{lm}

\textbf{Proof}
Applying H\"older, Young and Gagliardo--Nirenberg inequalities and using the charge conservation laws \eqref{cha-spl} and \eqref{cn-cha} and Jensen inequality, we have 
\begin{align*}
&\left\|\exp\left(2\tau\sum_{m\in\Z_{M+1}}\left\|u_\tau(t_m)\right\|_{L_{\infty}}\left\|u_m\right\|_{L_\infty}\right)\right\|_{\mathbb L^p(\Omega)}\\
&\le \exp\left(4pT^2e^{\alpha T}\|\xi\|^2\right)\prod_{m\in \Z_{M+1}} \left\|\exp\left(\frac{\left\|\nabla u_\tau(t_m)\right\|^2\tau}
{2pTe^{\alpha T}} \right)\right\|_{\mathbb L^{2(M+2)p}(\Omega)} \\
&\quad \times \prod_{m\in \Z_{M+1}} \left\|\exp\left(\frac{\left\|\nabla u_m\right\|^2\tau} {2pTe^{\alpha T}} \right)\right\|_{\mathbb L^{2(M+2)p}(\Omega)}\\
&\le C\left(\sup_{m\in \Z_{M+1}} \E\left[\exp\left(\frac {\left\|\nabla u_m\right\|^2} {e^{\alpha T}}\right)\right]
\cdot 
\sup_{t\in [0,T]} \E\left[\exp\left(\frac {\left\|\nabla u_\tau(t)\right\|^2}{e^{\alpha T}}\right)\right]  \right)^{\frac 1{2p}},
\end{align*}
where $\alpha$ is the parameter appearing in  \eqref{exp-ut0} or \eqref{exp-um1}. 
We complete the proof by the exponential integrability of $u_\tau$ and $u_m$ in Lemmas \ref{exp-ut} and \ref{exp-um}, respectively.
\qed
For the convenience, 
we can also define the continuous extension of $u_m$ as 
\begin{align*}
\widehat u_{\tau}(t):=\widehat u_{\tau,m}^S(t):=(\Phi_{j,t-t_m}^S\widehat {\Phi_{j,\tau}^D})\prod_{j=1}^{m-1}\big(\Phi_{j,\tau}^S\widehat {\Phi_{j,\tau}^D}\big)u_\tau(0),\quad t\in T_m,  
\end{align*}
where $\widehat {\Phi_{j,\tau}^D}$ is the solution operator of 
the Crank--Nicolson scheme.

\begin{tm}\label{u-um}
For any $p\ge 1$, there exist a constant $C=C(\xi,Q,T,p)$ such that
\begin{align}\label{u-um0}
\left(\E \left[\sup_{m\in \Z_{M}} \|u(t_m)-u_m\|^p \right]\right)^\frac1p
\le C \tau^\frac12.
\end{align}
\end{tm}
\textbf{Proof}
	We only prove the case $p=2$, since the proof for other cases is similar.
The strong error can be split as 
\begin{align*}
&\E \left[\sup_{m\in \Z_{M}} \|u(t_{m})-u_{m}\|^{2}\right] \\
&\le 2 \E \left[\sup_{m\in \Z_{M}} \|u(t_{m})-u_\tau(t_{m})\|^{2}\right]
+2\E \left[\sup_{m\in \Z_{M}} \|u_\tau(t_{m})-u_{m}\|^{2}\right].
\end{align*}
	Denote by $\widehat  e_{m+1}:=u(t_{m+1})-u_{m+1}$, $m\in \Z_{M-1}$. Applying It\^o formula to $\|\widehat e_{m+1}\|^2$ in $T_m$,  we obtain 
\begin{align*}
&\|\widehat e_{m+1}\|^2 
=\|u^{S}_{\tau, m}(t_{m})-\widehat u^{S}_{\tau, m}(t_{m})\|^{2}\\
&=\|\widehat e_m\|^{2}
+2\<\widehat e_m, \bi \int_{t_m}^{t_{m+1}}  \Delta u_\tau^{D}(s)-\Delta u^{D}_{m+\frac12}ds\>\\
&\quad+2\<\widehat e_m, \bi \lambda\int_{t_m}^{t_{m+1}}|u_{\tau, m}^{D}|^2u_{\tau, m}^{D}-\frac {|u_m|^2+|u^{D}_{m+1}|^2}2 u^{D}_{m+\frac12}ds\>
\\
&\quad+\bigg\|\int_{t_m}^{t_{m+1}}  
\Delta\left[u^{D}_{\tau, m}-u^{D}_{m+\frac12}\right]
+|u^{D}_\tau|^2u^{D}_\tau
-\frac {|  u_m|^2+|u^{D}_{m+1}|^2}2 u^{D}_{m+\frac12}ds\bigg\|^2\\
&:=\|\widehat e_m\|^{2}+III^m_1+III^m_2+III^m_3.
\end{align*}
Integration by parts, H\"older and Sobolev inequalities yield that 
	\begin{align*}
	III^m_1&=\left\<\Delta \widehat e_m,- 2\int_{t_m}^{t_{m+1}}\int_{t_m}^s \Delta u^{D}_{\tau, m}(r) + \lambda| u^{D}_{\tau, m}(r)|^2u^{D}_{\tau, m}(r)   dr ds\right\> \\
	&\quad+2\tau^2\left\<\Delta \widehat e_m,  \Delta u^D_{m+\frac12}
	+ \frac {|  u_m|^2+|u^{D}_{m+1}|^2}2 u^{D}_{m+\frac12}\right\>\\
	&\le C\tau^2\Big(\sup_{t\in T_M}\|u_{\tau, m}^D(t)\|^2_{\HH^2}
	+\|u_m\|_{\HH^2}^2+\|u_{m+1}^D\|^2_{\HH^2} \Big).
	\end{align*}
	For the term $III^m_2$, using cubic difference formula, Gagliardo--Nirenberg inequality, the charge and energy conservation law of Eq. \eqref{nls-d} and  Lemma \ref{discon}, we have 
	\begin{align*}
	III^m_2&=
	2\left\< \widehat e_m,\bi \lambda\int_{t_m}^{t_{m+1}} |u_{\tau, m}^{D}(s)|^2u_{\tau, m}^{D}(s)-|u^{D}_{\tau, m}(t_m)|^2u^{D}_{\tau, m}(t_m) ds\right\>\\
	&\quad+
	2\left\< \widehat e_m,\bi \lambda\int_{t_m}^{t_{m+1}}|u^{D}_{\tau, m}(t_m)|^2 u^{D}_{\tau, m}(t_m)-|  u_m|^2  u_m ds\right\>\\
	&\quad+2\left\< \widehat e_m,\bi \lambda\int_{t_m}^{t_{m+1}} |  u_m|^2  u_m-\frac {|  u_m|^2+|u^{D}_{m+1}|^2}2 u^{D}_{m+\frac12}ds\right\>\\ 
	&\le 2\tau \|\widehat e_m\|^2+C\int_{t_m}^{t_{m+1}}\left\| u_{\tau, m}^{D}(s)\right\|_{\HH^1}^2\left\|u_\tau^{D}(s)- u_{\tau, m}^{D}(t_m)\right\|^2ds\\
	&\quad+2\tau \left\|u_{\tau, m}^{D}(t_m)\right\|_{L_{\infty}}\left\|u_m\right\|_{L_\infty}\|\widehat e_m\|^2\\
	&\quad+C\tau\left(\left\|  u_m\right\|_{\HH^1}^2+\left\| u^{D}_{m+1}\right\|_{\HH^1}^2\right)\left\|u^{D}_{m+1}- u_{m}\right\|^2 \\
&\le \tau(2+2\left\|u_\tau(t_m)\right\|_{L_{\infty}}\left\|u_m\right\|_{L_\infty})\|\widehat e_m\|^2\\
	&\quad+C\tau^3\left(1+\sup_{t\in T_m}\| u^{D}_\tau(r)\|_{\HH^2}^4
	+\|u^{D}_{m+1} \|_{\HH^2}^6+\|u_{m} \|_{\HH^2}^6\right).
	\end{align*}
Analogously, 
	\begin{align*}
	III^m_3
	&\le C\tau^2 \Big(\sup_{t\in T_m}\| u^D_{\tau, m}(t)\|_{\HH^2}^2+\| u_m\|_{\HH^2}^2+\| u_{m+1}^D\|_{\HH^2}^2\Big).
	\end{align*}		
Summing up the estimations of $III^m_1$-$III^m_3$, we get
\begin{align*}
\|\widehat e_{m+1}\|^2
&\le \|\widehat e_m\|^{2}+\tau(2+2\left\|u_\tau(t_m)\right\|_{L_{\infty}}\left\|u_m\right\|_{L_\infty})\|\widehat e_m\|^2\\
&\quad+C\tau^2 \left(1+\sup_{t\in T_m}\| u^{D}_{\tau, m}(t)\|_{\HH^2}^4
+\|u^{D}_{m+1} \|_{\HH^2}^6+\|u_{m} \|_{\HH^2}^6\right).
\end{align*}
Applying Lemma \ref{dis-gro}, we have
\begin{align*}
&\|\widehat e_{m+1}\|^2 
\le C\tau\exp\left(2T+2\tau\sum_{n=0}^m \left\|u_\tau(t_n)\right\|_{L_{\infty}}\left\|u_n\right\|_{L_\infty}\right)  \\
&\quad \times \left(1+\sup_{t\in [0,T]}\| u^{D}_\tau(t)\|_{\HH^2}^4
+\sup_{m\in \Z_M}\|u^{D}_{m+1} \|_{\HH^2}^6+\sup_{m\in \Z_{M+1}}\|u_{m} \|_{\HH^2}^6\right).
\end{align*}
Then taking expectations on both sides and using H\"older inequality, we obtain
\begin{align*}
&\E \left[\sup_{m\in \Z_{M}}\|e_{m+1}\|^2\right]  \\
&\le C\tau\left\|\exp\left(2\tau\sum_{m\in\Z_{M}}\left\|u_\tau(t_m)\right\|_{L_{\infty}}\left\|u_m\right\|_{L_\infty}\right)\right\|_{\mathbb L^2(\Omega)}\\
&\quad \times \Bigg(1+\left(\E\left[\sup_{t\in [0,T]}\|u_\tau^D(t)\|_{\HH^2}^{8}\right] \right)^{\frac 12}
+\left(\E\left[\sup_{m\in \Z_M}\|u_{m+1}^D\|_{\HH^2}^{12}\right]\right)^{\frac 12} \\
&\qquad\quad +\left(\E\left[\sup_{m\in \Z_{M}}\|u_{m}\|_{\HH^2}^{12}\right] \right)^{\frac 12}\Bigg). 
\end{align*}
We conclude \eqref{u-um} for $p=2$ by combining Lemmas \ref{h2-um} and \ref{uutexp} and Corollaries \ref{h2-ud} and \ref{h2-umd}.
\qed

\section{Full Discretizations}
\label{sec-ful}

In this section, we first discretize Eq. \eqref{nls} in space by spectral Galerkin method and then apply the splitting Crank--Nicolson scheme 
in Section \ref{sec-cn} to the spatially discrete equation. 
Our main goal is to derive the strong error estimate of this
fully  discrete scheme.

\subsection{Spatial Spectral Galerkin Approximations}

We use the spectral Galerkin method to spatially discretize Eq. \eqref{nls} in this part and analyze its strong convergence rate.

Let $V_N$ be the subspace of $\mathbb L^2$ consists of the first $N$ eigenvectors of Dirichlet Laplacian operator.  
Denote by $\PPP^N: \mathbb L^ 2 \rightarrow V_N$ the spectral Galerkin projection defined by
$\<\PPP^N u,v\>=\<u,v\>$ for any $u\in \mathbb L^2$ and $v\in V_N$. 
It is clear that $\PPP^N$ is a self-adjoint and idempotent operator, i.e., 
\begin{align}\label{pn}
(\PPP^N)^*=\PPP^N,\quad (\PPP^N)^2=\PPP^N.
\end{align}
In order to inherit the charge conservation law, we directly use the spectral Galerkin method to approximate  Eq. \eqref{nls}.
The  corresponding numerical solution $u^N=\{u^N(t):\ t\in [0,T]\}$, $N\in \N$ satisfies   
\begin{align}\label{spe}
du^N&=\bi \left(\Delta u^N+\lambda \PPP^N (|u^N|^2u^N) 
-\PPP^N(u^N\circ dW(t))\right) , 
\quad u^N(0)=\PPP^N \xi.
\end{align}
We remark that the above approximation  is different form  applying the spectral Galerkin to approximate its equivalent It\^o
type SNLS eqution
\begin{align}\label{spe-i}
d\widetilde u^N&=\left(\bi \Delta \widetilde u^N+\bi \lambda \PPP^N (|\widetilde u^N|^2\widetilde u^N) 
-\frac 12 \PPP^N(F_Q\widetilde u^N)\right)dt\\\nonumber
&\quad-\bi\PPP^N(\widetilde u^N dW(t)),
\qquad \widetilde u^N(0)=\PPP^N \xi.
\end{align}
The following lemma shows that,
contrary to the deterministic NLS equation, i.e., Eq. \eqref{nls} with $Q=0$,
Eq. \eqref{spe-i} does not possess the charge conservation law. 

\begin{lm}\label{spech}
The charge of $u^N$	is  conserved, i.e., $\|u^N(t)\|=\|u^N(0)\|$ a.s.
The charge of $\widetilde u^N$ is not conserved, and satisfies the following evolution:
\begin{align}\label{cha-un}
\|\widetilde u^N(t)\|^2
=\|\widetilde u^N(0)\|^2-\int_0^t\sum_{k\in \N}\|(I-\PPP^N)(\widetilde u^NQ^{\frac12}e_k)\|^2 dr.
\end{align}
In particular, the charge of $u^N$ decreases, i.e.,
\begin{align}\label{cha-dec}
\|\widetilde u^N(t)\|^2\le \|\widetilde u^N(0)\|^2,\quad \text{a.s.}
\end{align}
\end{lm}

\text{Proof}
The conservation of $\|u^N\|$ is immediately obtained since the chain formula.
However, the spatial approximation applied to It\^o formula is totally different.		
Applying It\^o formula to the functional $\frac12 \|\widetilde u^N\|^2$ and using the fact \eqref{pn}, we have 
\begin{align*}
&\frac12\|\widetilde u^N(t)\|^2-\frac12\|\widetilde u^N(0)\|^2\\
&=\int_0^t\<\widetilde u^N,\bi \PPP^N(\Delta \widetilde u^N)\>dr
+\int_0^t\<\widetilde u^N,\bi \lambda \PPP^N(|\widetilde u^N|^2 \widetilde u^N)\>dr\\
&\quad-\int_0^t\<\widetilde u^N,\bi \PPP^N(\widetilde u^NdW(r))\>
-\frac12\int_0^t\<\widetilde u^N,\PPP^N(\widetilde u^NF_Q)\>	dr\\
&\quad +\frac12\int_0^t\sum_{k\in \N}\|\PPP^N(\widetilde u^NQ^{\frac12}e_k)\|^2dr\\
&=-\frac12\int_0^t\sum_{k\in \N}\|(I-\PPP^N)(\widetilde u^NQ^{\frac12}e_k)\|^2 dr
\le 0.
\end{align*}
This completes the proof of \eqref{cha-un} and \eqref{cha-dec}.
\\\qed

Based on  the charge evolution  of $u^N$ and $\widetilde u^N$, we have the following $\HH^2$-a priori estimates of $u^N$ and  $\widetilde u^N$, which implies the well-posedness of the spectral Galerkin method.

\begin{lm}\label{h2-un}
For any $p\ge 1$, there exists a constant $C=C(\xi,Q, T, p)$ such that 
\begin{align}\label{h2-un0}
\E\left[\sup_{t\in[0,T]} \|u^N(t)\|_{\HH^2}^p  \right]\le C,\quad \E\left[\sup_{t\in[0,T]} \|\widetilde u^N(t)\|_{\HH^2}^p  \right]\le C
\end{align}
\end{lm}

\text{Proof}
The proof of \eqref{h2-un0} is similar to that of \eqref{sta-u}.
We refer to \cite[Theorem 2.1]{CHL16b} for  details.
\qed

To deduce the strong convergence rate of the spectral Galerkin approximations \eqref{spe}, we also need the following exponential integrability of $u^N$ and $\widetilde u^N$ by applying Lemma \ref{exp-int} applied to $H(u^N)$ and  $\widetilde u^N$, whose proof is similar to that of \cite[Proposition 3.1]{CHL16b} or Lemma \eqref{exp-ut} and we omit the proof.

\begin{prop}\label{exp-un}
There exist some positive constants $\alpha$ and  $C$ depending on $\xi,Q$ and $T$ such that 
\begin{align}\label{exp-un0}
\sup_{t\in[0,T]}\E\left[\exp\left(\frac{\|\nabla u^N(t)\|^2}
{e^{\alpha t}} \right)\right]
&\le C,\quad \sup_{t\in[0,T]}\E\left[\exp\left(\frac{\|\nabla \widetilde u^N(t)\|^2}
{e^{\alpha t}} \right)\right]
&\le C.
\end{align}
\end{prop}

Now we are in position of estimating the strong convergence rate between the exact solution $u$ and the spectral Galerkin solution $u^N$, $\widetilde u^N$.
 Recall the following frequently used estimate for spectral Galerkin projection: 
 \begin{align}\label{spe-err}
 \|(I-\PPP^N) v\|\le  \lambda_{N+1}^{-1}\|v\|_{\HH^2},
 \quad v\in \HH^2,
 \end{align}
 where $\lambda_N$ is the $N$-th eigenvalue of Dirichlet negative Laplacian: $\lambda_N=(N\pi)^2$, $N\in \N_+$.

\begin{tm}\label{u-un}
For any $p\ge 1$, there exists a constant $C=C(\xi,Q, T, p)$ such that 
\begin{align}\label{u-un0}
\left(\E \left[\sup_{t\in [0,T]}\|u(t)-u^N(t)\|^p\right]\right)^\frac1p
&\le C N^{-2},\\
 \left(\E \left[\sup_{t\in [0,T]}\|u(t)-\widetilde u^N(t)\|^p\right]\right)^\frac1p
&\le C N^{-2}.
\end{align}
\end{tm}

\textbf{Proof}	
For simplicity, we only prove the case $p=2$ for $\widetilde u^N(t)$, the proof for the other case and $u^N$ is similar.
Denote by $\epsilon^N:=u-\widetilde u^N$.
Subtracting Eq. \eqref{spe} from Eq. \eqref{nls} yields that
\begin{align*}	
d \epsilon^N
&=\bi \Delta \epsilon^N dt
+\bi \lambda (|u|^{2}u-\PPP^N(|\widetilde u^N|^{2}\widetilde u^N))dt\\
&\quad -\frac12 \left[ uF_Q-\PPP^N(\widetilde u^NF_Q)\right] dt
-\bi\left[udW(t)-\PPP^N(\widetilde u^NdW(t))\right].
\end{align*}
Applying It\^o formula to the functional $\frac12\|\epsilon^N\|^2$, we get 
\begin{align*}
&\frac12\|\epsilon^N(t)\|^2-\frac12\|(I-\PPP^N) \xi\|^2\\
&=\int_0^t \<\epsilon^N,\bi(\Delta u^N-\PPP^N \Delta \widetilde u^N)\> dr
+\int_0^t \<\epsilon^N,\bi \lambda (|u|^{2}u-\PPP^N(|\widetilde u^N|^{2}\widetilde u^N))\> dr
\\
&\quad -\frac12\int_0^t \<\epsilon^N,(uF_Q-\PPP^N(\widetilde u^NF_Q))\> dr
-\int_0^t \<\epsilon^N,\bi(udW-\PPP^N(\widetilde u^NdW(r)))\> \\
&\quad+\frac12\int_0^t\sum_{k\in \N} \left(\|(I-\PPP^N)(u^NQ^{\frac12}e_k)\|^2+\|\PPP^N(\epsilon^NQ^{\frac12}e_k)\|^2\right) dr \\
&:=I^N_1+I^N_2+I^N_3+I^N_4+I^N_5.
\end{align*}
Due to the property of  the spectral Galerkin projection operator $\PPP^N$, $IV_1=0$.
By the identity $|a|^{2}a-|b|^{2}b=(|a|^{2}+|b|^2)(a-b)+ab(\overline {a-b})$ for $a,b \in \mathbb C$ and the estimation \eqref{spe-err}, we have 
\begin{align*}
I^N_2
&=\int_0^t\<\epsilon^N,\bi \lambda(|u|^{2}u-|\widetilde u^N|^{2}\widetilde u^N)\> 
+\<\epsilon^N,\bi \lambda(I-\PPP^N)(|\widetilde u^N|^{2}\widetilde u^N)\> dr\\
&\le \int_0^t \|\epsilon^N\|^2 \|u\|_{L_\infty} \|\widetilde u^N\|_{L_\infty} dr
+\int_0^t\|\epsilon^N\|\|(I-\PPP^N)(|\widetilde u^N|^{2}\widetilde u^N)\| dr  \\
&\le \frac12\lambda_{N+1}^{-2} \int_0^t \|\widetilde u^N\|_{\HH^2}^{6} dr
+\int_0^t  \left(\frac12+\|u\|_{L_\infty} \|\widetilde u^N\|_{L_\infty}\right) 
\|\epsilon^N\|^2 dr.
\end{align*}
It follows from \eqref{spe-err} and Cauchy-Schwarz inequality that
\begin{align*}
I^N_3+I^N_5
&=-\frac12 \int_0^t\sum_k\|(I-\PPP^N)(\epsilon^NQ^{\frac12}e_k)\|^2 dr\\
&\quad -\frac12 \int_0^t \<\epsilon^N,(I-\PPP^N)(\widetilde u^NF_Q)\>dr\\
&\quad+\frac12\int_0^t \sum_k\|(I-\PPP^N)(uQ^{\frac12}e_k)\|^2dr\\
&\le \frac {\|Q^{\frac12}\|_{\LL_2^2}^2}2\int_0^t \left( \|\epsilon^N\|^2+\lambda_{N+1}^{-2}\|u\|_{\HH^{2}}^2
+\lambda_{N+1}^{-2}\|\widetilde u^N\|_{\HH^{2}}^2\right)dr.
\end{align*}
By the properties of $\PPP^N$, we can rewrite the third term $I^N_3$ as
\begin{align*}
I^N_4=-\int_0^t \<(I-\PPP^N)u,\bi(I-\PPP^N)(\widetilde u^NdW(r))\>.
\end{align*}
Combining the above  estimations, we obtain
\begin{align*}
\|\epsilon^N(t)\|^2
&\le \lambda_{N+1}^{-2} \left(\|\xi\|^2
+\int_0^t \left[ \|Q^\frac12\|_{\LL_2^2}^2 (\|\widetilde u^N\|_{\HH^2}^2+\|u\|_{\HH^2}^2 )
+\|\widetilde u^N\|_{\HH^2}^{6}\right] dr\right)  \\
&\quad +2\left|\int_0^t \<(I-\PPP^N)u,\bi(I-\PPP^N)(\widetilde u^NdW(r))\>\right| \\
&\quad +\int_0^t  \left(1+\|Q^\frac12\|_{\LL_2^2}^2 +2\|u\|_{L_\infty}\|\widetilde u^N\|_{L_\infty}\right) 
\|\epsilon^N\|^2 dr.
\end{align*}
Gronwall inequality implies that 
\begin{align*}
&\|\epsilon^N(t)\|^2
\le \exp\left(\int_0^T 1+\|Q^\frac12\|_{\LL_2^2}^2 +2\|u\|_{L_\infty} \|\widetilde u^N\|_{L_\infty} dr\right)\\
&\quad \times \bigg(\lambda_{N+1}^{-2} \left(\|\xi\|^2
+\int_0^T \left[ \|Q^\frac12\|_{\LL_2^2}^2 (\|\widetilde u^N\|_{\HH^2}^2+\|u\|_{\HH^2}^2) 
+\|u^N\|_{\HH^2}^{6}\right] dr\right)  \\
&\qquad +2\sup_{t\in[0,T]} \left|\int_0^t \<(I-\PPP^N)u,\bi(I-\PPP^N)(\widetilde u^NdW(r))\>\right|\bigg).
\end{align*}
Now taking the supreme over $t$, taking expectation on both sides in the above inequality and using the a priori estimation \eqref{h2-un0} in Lemma \ref{h2-un} as well as the exponential integrability \eqref{exp-un0} in Lemma \ref{exp-un}, we obtain
\begin{align*}
\E\left[\sup_{t\in[0,T]} \|\epsilon^N(t)\|^2\right]  
\le C\left(\lambda_{N+1}^{-2}+R^N\right),
\end{align*}	
where 
\begin{align*}
R^N:=\left\|
\int_0^t \<(I-\PPP^N) u,\bi(I-\PPP^N)(\widetilde u^NdW(r))\>\right\|_{\mathbb L^2(\Omega; \mathbb L^{\infty}([0,T]))},
\end{align*}	
It suffices to estimate the stochastic integral term appeared above.
This can be done by Burkholder--Davis--Gundy inequality and the estimation \eqref{spe-err}:
\begin{align*}
R^N 
&\le C\left(\sum_{k=1}^{\infty}\int_0^t
\E \left[\|(I-\PPP^N)u\|^2 \|(I-\PPP^N)\widetilde u^NQ^{\frac12}e_k\|^2 \right] dr \right)^\frac12 \\
&\le C\|Q^{\frac12}\|_{\LL_2^2} \lambda_{N+1}^{-2}\left(\int_0^t \E \left[\|u\|_{\HH^2}^2\|\widetilde u^N\|_{\HH^2}^2 \right]ds\right)^\frac12
\le C\lambda_{N+1}^{-2}.
\end{align*}
This completes the proof of \eqref{u-un0}.
\\\qed

\begin{rk}\label{rk-fem}
For general finite element methods, the projection operator $\PPP^N$ does not commute with $\Delta$ in the term $\<\epsilon^N,\bi(\Delta u^N-\PPP^N \Delta u^N)\>$. As a result,  this term produce $\| (I-\PPP^N) \Delta u^N\|$ which requires more regularity.
Compared to finite element methods, the spectral Galerkin method can achieve the optimal order. 
\end{rk}

\subsection{Spectral splitting Crank--Nicolson Scheme}

Motivated by  Theorem \ref{u-um} and Theorem \ref{u-un}, we propose a fully discrete splitting Crank--Nicolson scheme, in this part, to inherit the charge conservation law,
stability and exponential integrability.
As in Section \ref{sec-spl}, we split the spatially semi-discrete spectral approximations \eqref{spe} into
\begin{align}\label{spe-spl}
\begin{cases}
du^{(N,D)}=\left(\bi \Delta u^{(N,D)}
+\bi \lambda \PPP^N \left(|u^{(N,D)}|^{2} u^{(N,D)}\right) \right)dt, \\
du^{(N,S)}=-\bi \PPP^N (u^{(N,S)} \circ dW(t)).
\end{cases}
\end{align}
Then we apply the temporal splitting Crank--Nicolson scheme \eqref{spl-cn} to the above spectral splitting systems \eqref{spe-spl}, which leads to the following fully discrete spectral splitting Crank--Nicolson scheme:
\begin{align}\label{spe-cn}
\begin{cases}
u^{(N,D)}_{m+1}=u^N_m+ \bi \tau \Delta  u^{(N,D)}_{m+\frac12}
+\frac{\bi \lambda \tau}2
\PPP^N \left( \left(| u^{N}_m|^2+|u^{(N,D)}_{m+1}|^2\right) 
u^{(N,D)}_{m+\frac12}\right), \\
u_{m+1}^N= \exp\left(-\bi\PPP^N(W(t_{m+1})-W(t_m))\right) \cdot u^{(N,D)}_{m+1}.
\end{cases}
\end{align}
where $ u^{(N,D)}_{m+\frac12}=\frac12 ( u^{N}_m + u^{(N,D)}_{m+1})$.

It is not difficult to see that the charge of 
Eq. \eqref{spe-spl} is conserved since both spatial and temporal numerical methods preserves the charge. 
Similarly to the proof of Lemma \eqref{exp-ut} and Lemma \ref{h2-um}, we have the $\HH^1$-exponential integrability and $\HH^2$-a priori estimate of  $u_m^N$, $m\in \Z_M$.

\begin{lm}\label{umn}
Let $p\ge 1$.
There exist constants $C$ and $\alpha$ depending on $\xi$, $Q$ and $T$ and 
$C'=C'(\xi,Q,T,p)$ such that 
\begin{align*}
\sup_{m\in \Z_{M}}\E \left[\exp\left(\frac {\|\nabla u_m^N\|^2}{e^{\alpha t_m}}\right)\right]
\le C
\end{align*}
and
\begin{align*}
\E\left[\sup_{m\in \Z_{M}}\| u_m^N\|_{\HH^2}^p \right]\le C'.
\end{align*}
\end{lm}

As a consequence of the exponential integrability, we have the following Gaussian tail estimations for $u(t)$, $u_\tau(t)$, $u_m$, $u^N(t)$ and $u_m^N$, $t\in [0,T]$, $m\in \Z_{M}$ with $M\in \N$ and $N\in \N_+$, which have their own interest.

\begin{cor}\label{tail}
There exist  constants $C$ and $\eta$ depending on $\xi,Q$ and $T$ such that for any $x\in \R_+$,
\begin{align*}
&\sup_{t\in [0,T]}\PP\left(\|u(t)\|_{\HH^1}\ge x\right)
+ \sup_{t\in [0,T]}\PP\left(\|u_\tau(t)\|_{\HH^1}\ge x\right)
+\sup_{m\in \Z_{M}}\PP\left(\|u_m\|_{\HH^1}\ge x\right)\\
&\quad+\sup_{t\in [0,T]}\PP\left(\|u^N(t)\|_{\HH^1}\ge x\right)+
\sup_{m\in \Z_{M}}\PP\left(\|u^N_m\|_{\HH^1}\ge x\right)\le  C\exp(-\eta x^2).
\end{align*}
\end{cor}

\textbf{Proof}
We only prove the Gaussian tail estimation of $u$; the other cases are similar.
By Chebyshev inequality and the exponential integrability of $u$ in Theorem \ref{well}, we deduce that for any $y\ge 1$ there holds that 
\begin{align*}
\PP\left(\|u(t)\|_{\HH^1}\ge \sqrt{\ln(y)e^{\alpha T}}\right)
&=\PP\left(\exp\left(e^{-\alpha T}\|u(t)\|_{\HH^1}^2\right) \ge y\right)\\
&\le \frac{\E\left[\exp\left(e^{-\alpha T}\|u(t)\|_{\HH^1}^2\right)\right]}y.
\end{align*}
Let $x=\sqrt{\ln(y)e^{\alpha T}}$.
Then we obtain
\begin{align*}
\PP\left(\|u(t)\|_{\HH^1}\ge x\right)
&\le \E\left[\exp\left(e^{-\alpha T}\|u(t)\|_{\HH^1}^2\right)\right]  \exp\left(-e^{-\alpha T}x^2\right) \\
&\le  C\exp(-\eta x^2)
\end{align*}
for $C$ as in \eqref{exp-u} and $\eta=e^{-\alpha T}$.
\\\qed

Based on Theorems \ref{u-um} and \ref{u-un} and Lemma \ref{umn}, we derive the strong error between $u$ and $u_m^N$.

\begin{tm}\label{u-umn}
For any $p\ge 1$, there exists $C=C(\xi,Q,T,p)$ such that
\begin{align}\label{u-umn0}
\left(\E \left[\sup_{m\in \Z_m} \|u(t_m)-u_m^N\|^p \right]\right)^\frac1p
\le C (N^{-2}+ \tau^\frac12).
\end{align}
\end{tm}

\textbf{Proof}
We split the error into a spatial error and a temporal error:
\begin{align*}
\|u(t_m)-u_m^N\|^p
\le C(\|u(t_m)-u^N_\tau(t_m)\|^p
+\|u^N_\tau(t_m)- u_m^N\|^p).
\end{align*}
The spatial error is controlled by \eqref{u-un0} in Theorem \ref{u-un},
and the temporal error is estimated by \eqref{u-um0} in Theorem \ref{u-um} and the a priori estimations in Lemma \ref{umn}. 
\qed	

\subsection{Finite difference splitting Crank--Nicolson scheme}
The proposed temporal splitting approach gives a kind of full discretizations with different spatial approximations, such as the following finite difference splitting Crank--Nicolson scheme. 
Let $h=\frac1{N+1}$ and $\{x_n:=n h,\ n\in \Z_{N+1}\}$ be a partition of the spatial interval $\OOO$.
We define a grid function $f^h$ in $\{x_n\}_{n\in \Z_{N+1}}$ as $f^h(x_n)=f^h(n)$ with $f^h(0)=0$ and $f^h(N+1)=0$, $m\in \Z_{N+1}$.
Denote $\delta_+f^h(n):=\frac {f^h(n+1)-f^h(n)}h$ and $\delta_-f^h(n):=\frac {f^h(n)-f^h(n-1)}h$.
The finite difference splitting Crank--Nicolson scheme is
\begin{align}\label{fdm}
\begin{cases}
u^{(h,D)}_{m+1}=u^h_m+\bi \tau\left(\delta_+ \delta_-u^{(h,D)}_{m+\frac 12}
+\frac\lambda2 \left(| u^h_m|^2+|u^{(h,D)}_{m+1}|^2\right) u^{(h,D)}_{m+\frac12} \right),\\
u^h_{m+1}=\exp(-\bi(W(t_{m+1})-W(t_m))) u^{(h,D)}_{m+1},
\quad m\in \Z_M, 
\end{cases}
\end{align}
where $u^{(h,D)}_{m+\frac 12}=\frac 12(u^h_m+u^{(h,D)}_{m+1})$.

By combining the error estimate of spatial centered difference method in \cite[Theorem 4.1]{CHL16b} and similar arguments in Theorem \ref{u-umn}, the strong error order of the above finite difference splitting Crank--Nicolson scheme \eqref{fdm} can be obtained.

\begin{rk}
Combining our temporal splitting Crank--Nicolson scheme with spatially numerical methods such as Galerkin finite element method and certain finite difference methods, we can also obtain strong convergence rates of related fully discrete splitting schemes.
However, as noted in Remark \ref{rk-fem}, one may not obtain optimal strong convergence rates for these schemes under minimal regularity assumptions on initial datum $\xi$ and noise's covariance operator $Q$.
\end{rk}

\section{Numerical Experiments}
\label{sec-num}

In this section, we present several numerical tests to verify our theoretic results including the evolution of charge, the exponential moments of energy and the strong convergence rates for the numerical schemes.

We use the  spectral splitting Crank--Nicolson scheme \eqref{spe-cn}  to fully discretize the following stochastic NLS equation with a noise intensity $\ee\in \R$:
\begin{align*}
\begin{cases}
\bi\, du+(\Delta u+|u|^2 u) dt=\ee \, u\circ dW(t),
\quad (t,x)\in (0,T]\times (0,1);\\
u(t,0)=u(t,1)=0,\quad t\in [0,T];\\
u(0,x)=\text{sin}(\pi x),\quad x\in (0,1).
\end{cases}
\end{align*}
Here we take the $Q$-wiener process as 
\begin{align*}
W(t,x)=\sum_{k=1}^\infty \frac{\sqrt2 \sin(k\pi x)}{1+k^{2.6}} \beta_k(t),\quad (t,x)\in [0,T]\times (0,1).
\end{align*}
To simulate this process, we truncate the series by the first $K$ terms with various $K\in \N_+$ and take $P=1000$ trajectories.


\begin{figure}
	\centering
	\subfigure[]{
		\includegraphics[width=0.8\linewidth]{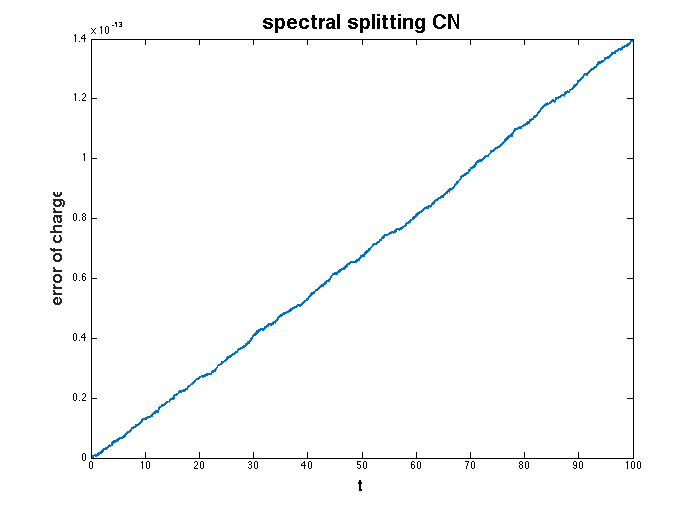}
	}
	\caption{Evolution of error for charge($\|u^N_m\|^2-\|\xi\|^2$) by spectral splitting Crank--Nicolson scheme with $K=N$.} 
\label{fig:charge}
\end{figure}

For the simulations of the evolution of charge and exponential moment of energy, we take $T=100$, $\ee = 10$, $N=2^6$,
$h=2^{-6}$ and $\tau=2^{-10}$.
Figure \ref{fig:charge} shows the charge conservation law of the spectral splitting Crank--Nicolson scheme \eqref{spe-cn}.
Figures \ref{fig:exp_spe}  illustrate the exponential moment of energy for schemes \eqref{spe-cn} and \eqref{fdm} with different parameters $\alpha = 0.7$ and  $\alpha=1$ appeared in Lemma \ref{exp-um}, which recover the exponential integrability results \eqref{exp-um0}.

\begin{figure}
	\centering
	\subfigure[]{
		\includegraphics[width=0.45\linewidth]{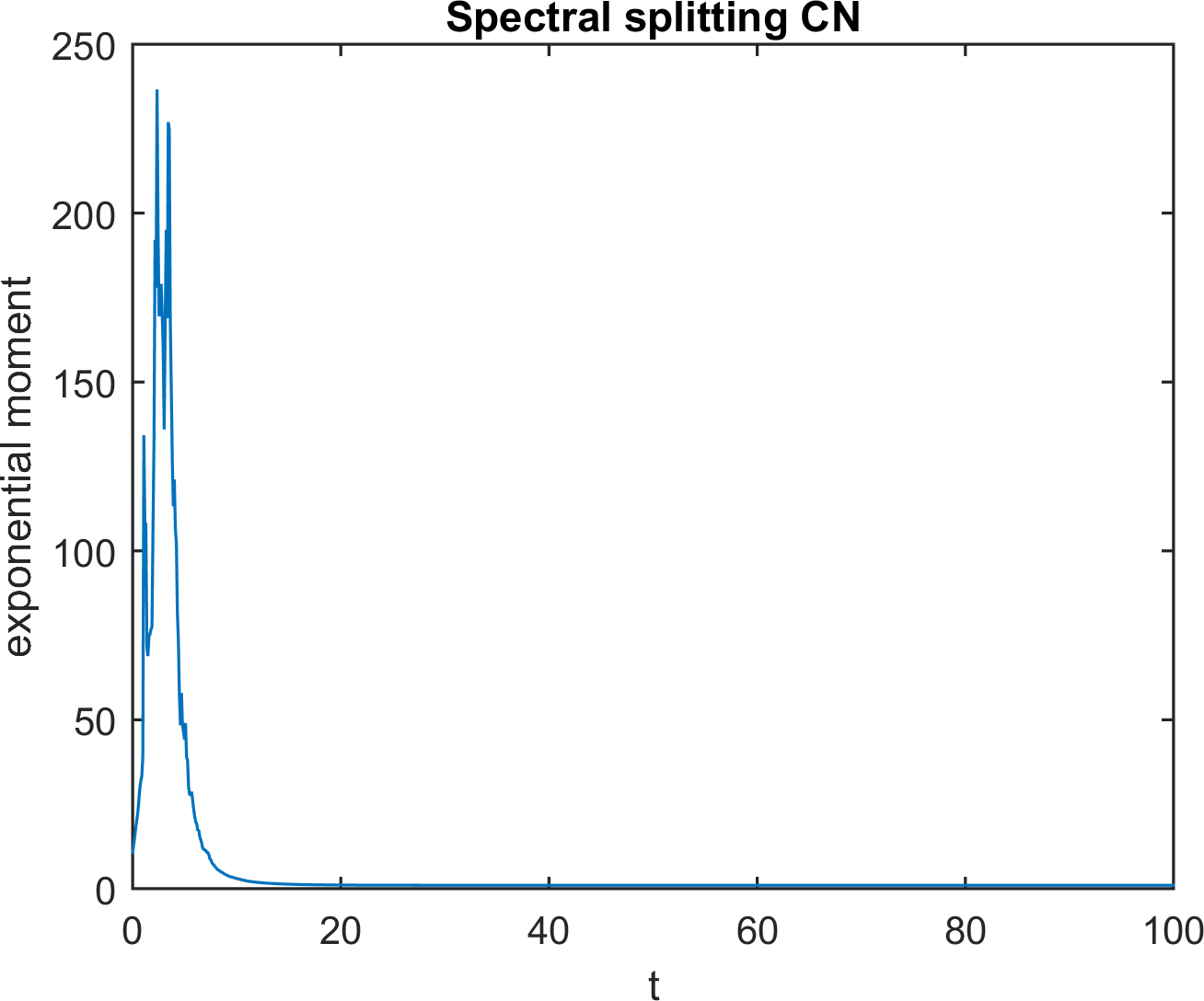}
	}
	\subfigure[]{
		\includegraphics[width=0.45\linewidth]{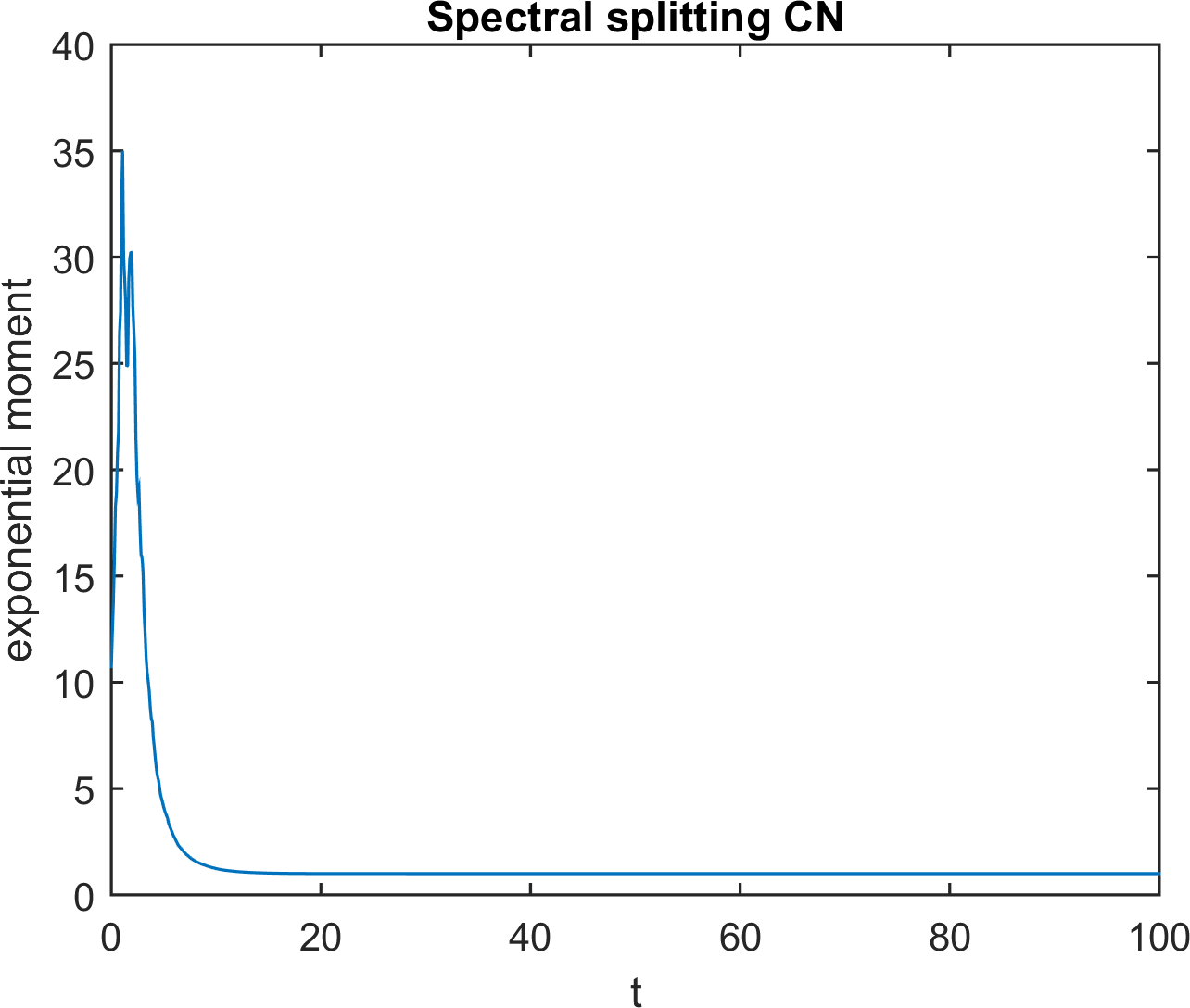}
	}
	\caption{Exponential integrability for spectral splitting Crank--Nicolson: $(a)$ $\alpha=0.7$ and  $(b)$ $\alpha=1$.}
	\label{fig:exp_spe}
\end{figure}

Next we turn to the tests about the temporal and spatial strong convergence rates of the schemes \eqref{spe-cn}. 
Errors of the numerical solutions against $\tau$ or $h$ on a log-log scale are displayed in Figures \ref{fig:spe}.
We also apply two noise intensities $\ee =1$ and $\ee=10$ in these tests to check the strong convergence results.
More precisely, the left figures (a) in Figures \ref{fig:spe} presents the temporal strong convergence rates for these two schemes.
Since there is no analytic solution for SPDEs with non-monotone nonlinearity in nearly all cases, we first compute a reference solution $u_{{ref}}$ on a fine mesh $\tau_{ref}=2^{-14}$. 
Compared to five coarser girds by $\tau=2^p\tau_{ref}$, $p=1,\cdots,5$, the strong errors at $T=1$ are plotted with $N=2^8$.
The right figures (b) in Figures \ref{fig:spe} shows the spatial strong convergence rates.
For fixed $\tau=2^{-8}$, the corresponding reference spatial mesh is $h_{ref}=2^{-10}$ and other five coarser girds are $h=2^ph_{ref}$, $p=1,\cdots,5$. The slopes in these figures indicate that the two schemes both possess temporal strong convergence order 1/2 and spatial strong convergence order 2, which coincides with the theoretical results in Theorems \ref{u-umn}.

\begin{figure}
	\centering
	\subfigure[]{
		\includegraphics[width=0.45\linewidth]{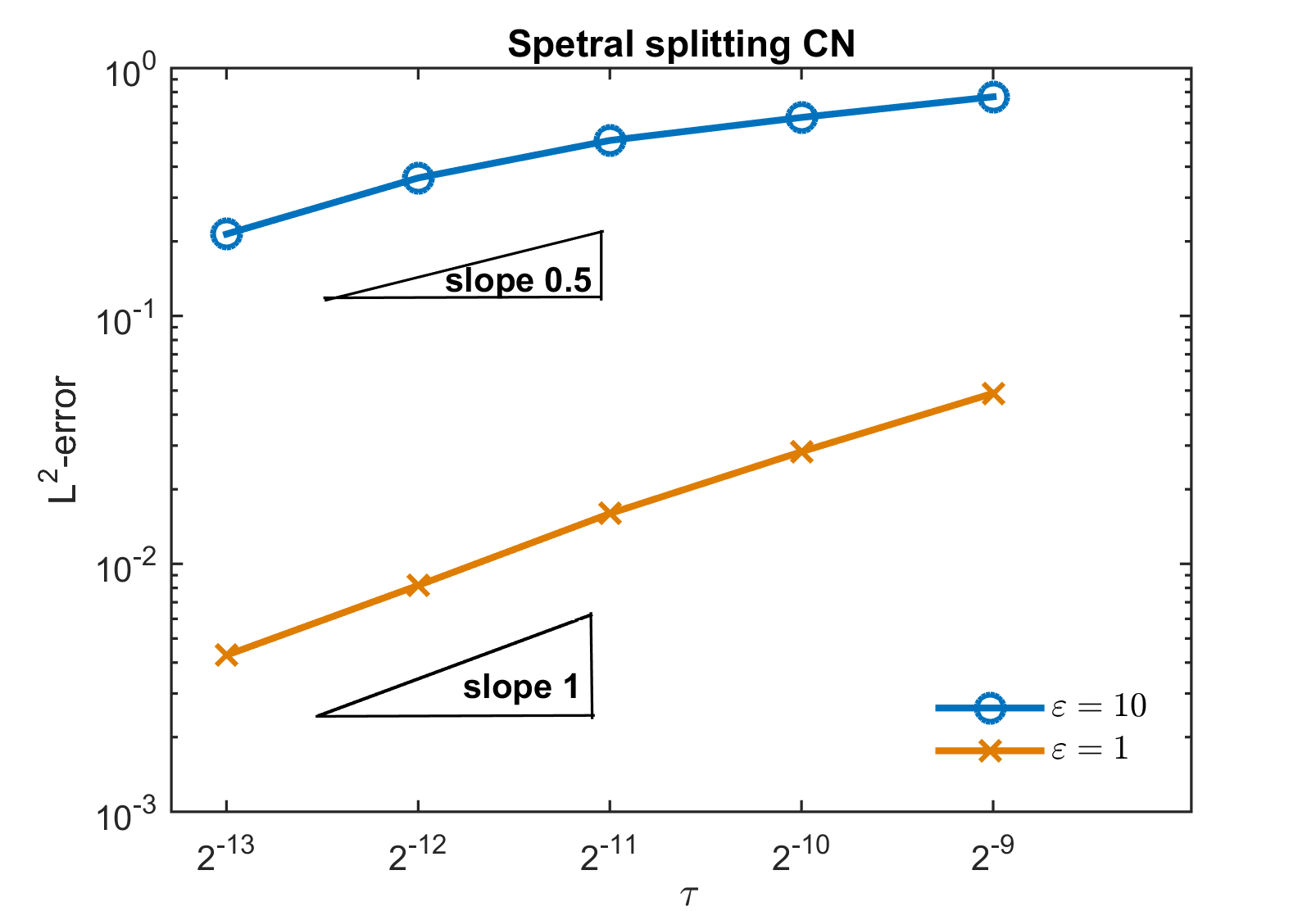}
	}
	\subfigure[]{
		\includegraphics[width=0.45\linewidth]{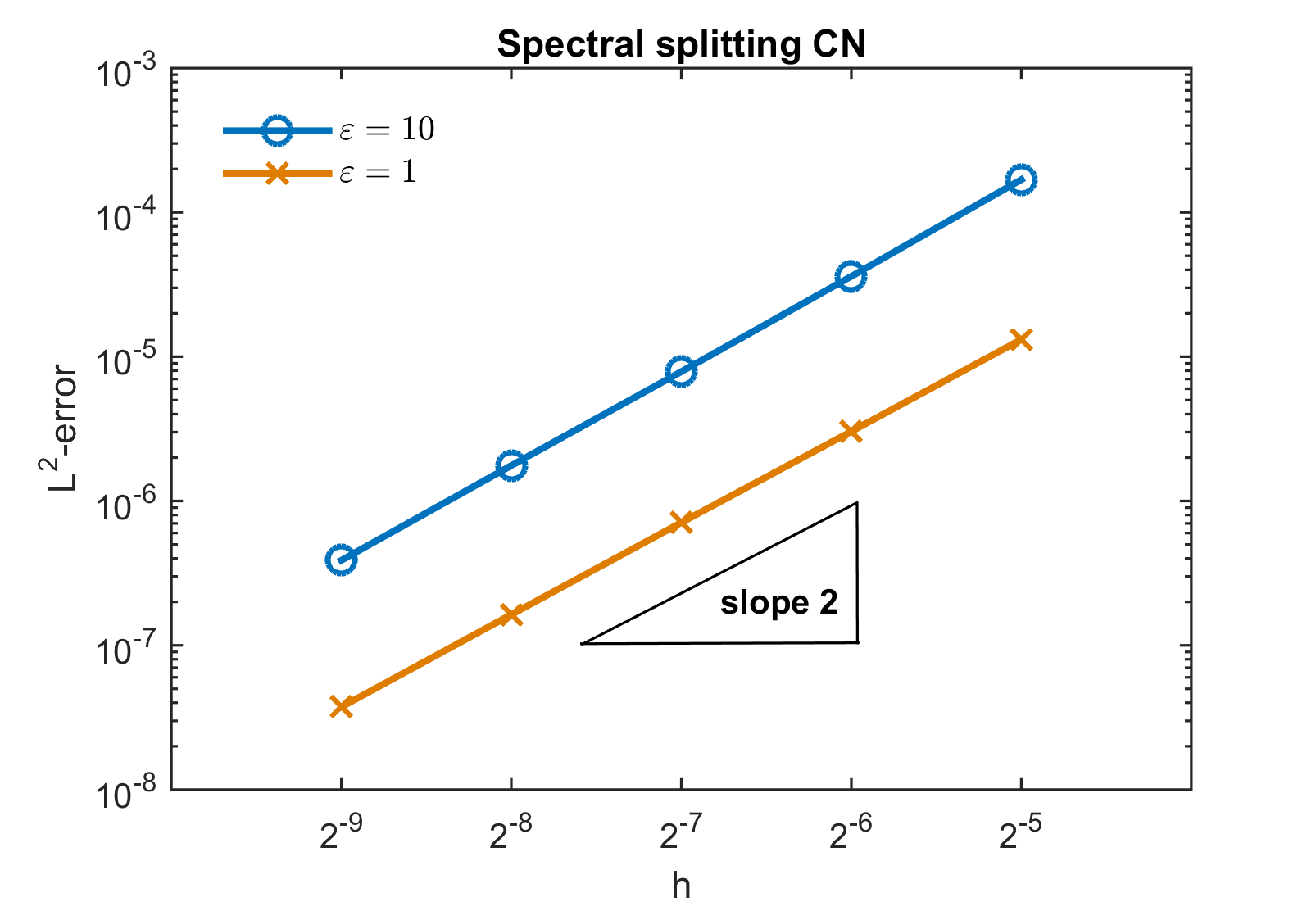}
	}
	\caption{Temporal and spatial convergence rates for pseudo-spectral splitting Crank--Nicolson:  $(a)$ temporal strong convergence rate and $(b)$ spatial convergence rate.}
\label{fig:spe}
\end{figure}

\bibliographystyle{amsalpha}
\bibliography{bib}

\end{document}